\documentclass[12pt,oneside]{article}
\usepackage{amsmath,amssymb,amsfonts,amsthm}
\newcommand{\T}{{\cal T}}

\newcommand{\Real}{\mathbb R}

\newcommand{\To}{\longrightarrow}

\newcommand{\p}{\pi^{-1}(TM)}
\newcommand {\cp}{\mathfrak{X}(\pi (M))}
\newcommand {\cpp}{\mathfrak{X}(\T M)}

\def\Section#1{\vspace{30truept}\addtocounter{section}{1}\setcounter{thm}{0}
\setcounter{equation}{0}{\noindent\Large\bf
    \arabic{section}.~~#1}\par \vspace{12pt}}

\newtheorem{thm}{Theorem}[section]
\newtheorem{cor}[thm]{Corollary}
\newtheorem{lem}[thm]{Lemma}
\newtheorem{prop}[thm]{Proposition}
\newtheorem{defn}[thm]{Definition}

\numberwithin{equation}{section}
\begin{document}

\title{ GEOMETRIC OBJECTS ASSOCIATED WITH THE FUNDAMENTAL CONNECTIONS IN FINSLER GEOMETRY }
\author{{\bf Nabil L. Youssef$^{\dag}$, S. H. Abed$^{\dag}$ and A. Soleiman$^{\ddag}$}}
\date{}
%\thanks{\it Department of Mathematics, etc}
%\pagestyle{fancy}

             % End of preamble and beginning of text.
\maketitle                     % Produces the title.
\vspace{-1.15cm}
\begin{center}
{$^{\dag}$Department of Mathematics, Faculty of Science,\\ Cairo
University, Giza, Egypt}
\end{center}
\vspace{-0.8cm}
\begin{center}
nlyoussef2003@yahoo.fr,\ sabed52@yahoo.fr
\end{center}
\vspace{-0.7cm}
\begin{center}
and
\end{center}
\vspace{-0.7cm}
\begin{center}
{$^{\ddag}$Department of Mathematics, Faculty of Science,\\ Benha
University, Benha,
 Egypt}
\end{center}
\vspace{-0.8cm}
\begin{center}
soleiman@mailer.eun.eg
\end{center}
\smallskip

\vspace{1cm} \maketitle
\smallskip

\noindent{\bf Abstract.}   The aim of the present paper is to
provide an \emph{intrinsic} investigation  of the properties of the
most important geometric objects associated with the fundamental
linear connections in Finsler geometry. We investigate intrinsically
the most general relations concerning the torsion tensor fields and
the curvature tensor fields associated with a given regular
connection on the pullback bundle of a Finsler manifold. These
relations, in turn,  play a key role in obtaining other interesting
results concerning the properties of the most important  geometric
objects associated with the fundamental canonical linear connections
on the pullback bundle of a Finsler manifold, namely, the Cartan
connection, the Berwald connection, the Chern (Rund) connection and
the Hashiguchi connection.
\par
For the sake of completeness and for comparison reasons, we provide
an appendix presenting a global survey of canonical linear
connections in Finsler geometry and the fundamental geometric
objects associated with them. 



%%%%%%%%%%%%%%%%%%%%%%%%%%%%%%%%%%% INTRODUCTION %%%%%%%%%%%%%%%%%%%%%%%%%%%%%%%%%%%%
\newpage
%\vspace{1cm}
%Introduction
\vspace{30truept}\centerline{\Large\bf{Introduction}}\vspace{12pt}
\par
Studying Finsler geometry, one encounters substantial difficulties
trying to seek analogues of classical global, or sometimes even
local, results of Riemannian geometry. These difficulties arise
mainly from the fact that in Finsler geometry all geometric objects
depend not only on positional coordinates, as in Riemannian
geometry, but also on directional arguments.
\par
 In Riemannian geometry, there is a canonical linear connection on the
manifold $M$, namely, the Levi-Civita connection, whereas in Finsler
geometry there is a corresponding canonical linear connection due to
E. Cartan. However, this is not a connection on $M$ but is a
connection on $T(\T M)$, the tangent bundle of $\,\T M$, or on
$\,\pi^{-1}(TM) $, the pullback of the tangent bundle $TM$ by
$\,\pi: \T M\longrightarrow M$. Moreover, in Riemannian geometry
there is one curvature tensor and no torsion tensor  associated with
the Levi-Civita connection on $M$, whereas  in Finsler geometry
 there are three curvature tensors  and five torsion tensors
 associated with the Cartan  connection on $\,\pi^{-1}(TM) $.
 Besides, there are other canonical linear connections together with their associated torsion and curvature tensor fields.
 Consequently, Finsler geometry is richer in structure and content than Riemannian
 geometry  and thus potentially
  more appropriate for dealing with physical theories at a deeper level.
\par
 The theory of connections and their associated geometric objects is an important field of differential
 geometry.  The most important linear connections and their associated  geometric objects  in Finsler geometry were
 studied \emph{locally} in \cite{r91}, \cite{r34}, \cite{r93},...etc.
\par
In \cite{r92} and  \cite{r94}, we have provided new intrinsic proofs
of intrinsic versions of the existence and uniqueness theorems for
the fundamental linear connections on the pullback bundle of a
Finsler manifold, namely,   the Cartan connection, the Berwald
connection, the Chern (Rund) connection and the Hashiguchi
connection. The present paper is a continuation of this work where
we investigate \emph{intrinsically} the fundamental properties of
the most important geometric objects associated with these
connections.
\par
The  paper consists of five parts preceded by an introductory
section $(\S 1)$, which provides a brief account of the basic
concepts and results necessary for this work. For more details, we
refer to \cite{r58}, \cite{r61}, \cite{r21}, \cite{r74}, and
\cite{r44}.
\par
In the first part $(\S 2)$, we investigate the fundamental relations
concerning the torsion tensor fields and the curvature tensor fields
associated with a given regular connection on $\pi^{-1}(TM)$. These
relations, in turn, play a key role in obtaining other interesting
results. The second part $(\S 3)$ is devoted to study the
fundamental properties of the most important geometric objects
associated with the Cartan connection. In the third part $(\S 4)$,
various fundamental relations and properties concerning the torsion
tensor fields and the curvature tensor fields associated with the
Berwald connection are obtained. In the fourth  and the fifth parts
($(\S 5)$ and $(\S 6)$), as in the previous sections, we study  the
most important geometric objects associated with the Chern
connection and the Hashiguchi connection, respectively.
\par
For the sake of completeness and for comparison reasons, the paper
is concluded with an appendix presenting a global survey of
canonical linear connections in Finsler geometry and the fundamental
geometric objects associated with them.
\par
The present work is formulated in a prospective modern
coordinate-free form, without being trapped into the complications
of indices. However, the local expressions of the obtained results,
when calculated, coincide with the existing classical local results.

%%%%%%%%%%%%%%%%%%%%%%%% SECTION 1. Notation and Preliminaries %%%%%%%%%%%%%%%%

\Section{Notation and Preliminaries}

In this section, we give a brief account of the basic concepts
 of the pullback approach to intrinsic Finsler geometry necessary for this work. For more
 details, we refer to \cite{r58},\,\cite{r61},\,\cite{r74} and~\,\cite{r44}.
 We make the
assumption that the geometric objects we consider are of class
$C^{\infty}$.\\ The
following notation will be used throughout this paper:\\
 $M$: a real paracompact differentiable manifold of finite dimension $n$ and of
class $C^{\infty}$,\\
 $\mathfrak{F}(M)$: the $\Real$-algebra of differentiable functions
on $M$,\\
 $\mathfrak{X}(M)$: the $\mathfrak{F}(M)$-module of vector fields
on $M$,\\
$\pi_{M}:TM\longrightarrow M$: the tangent bundle of $M$,\\
$\pi: \T M\longrightarrow M$: the subbundle of nonzero vectors
tangent to $M$,\\
$V(TM)$: the vertical subbundle of the bundle $TTM$,\\
 $P:\pi^{-1}(TM)\longrightarrow \T M$ : the pullback of the
tangent bundle $TM$ by $\pi$,\\
 $\mathfrak{X}(\pi (M))$: the $\mathfrak{F}(\T M)$-module of
differentiable sections of  $\pi^{-1}(T M)$,\\
$ i_{X}$ : the interior product with respect to  $X
\in\mathfrak{X}(M)$,\\
$df$ : the exterior derivative  of $f\in \mathfrak{X}(M)$,\\
$ d_{L}:=[i_{L},d]$, $i_{L}$ being the interior derivative with
respect to a vector form $L$.

\par Elements  of  $\mathfrak{X}(\pi (M))$ will be called
$\pi$-vector fields and will be denoted by barred letters
$\overline{X} $. Tensor fields on $\pi^{-1}(TM)$ will be called
$\pi$-tensor fields. The fundamental $\pi$-vector field is the
$\pi$-vector field $\overline{\eta}$ defined by
$\overline{\eta}(u)=(u,u)$ for all $u\in \T M$.

We have the following short exact sequence of vector bundles,
relating the tangent bundle $T(\T M)$ and the pullback bundle
$\pi^{-1}(TM)$:\vspace{-0.1cm}
$$0\longrightarrow
 \pi^{-1}(TM)\stackrel{\gamma}\longrightarrow T(\T M)\stackrel{\rho}\longrightarrow
\pi^{-1}(TM)\longrightarrow 0 ,\vspace{-0.1cm}$$
 where the bundle morphisms $\rho$ and $\gamma$ are defined respectively by
$\rho := (\pi_{\T M},d\pi)$ and $\gamma (u,v):=j_{u}(v)$, where
$j_{u}$  is the natural isomorphism $j_{u}:T_{\pi_{M}(v)}M
\longrightarrow T_{u}(T_{\pi_{M}(v)}M)$. The vector $1$-form $J$ on
$TM$ defined by $J:=\gamma\circ\rho$ is called the natural almost
tangent structure of $T M$. The vertical vector field $\mathcal{C}$
on $TM$ defined by $\mathcal{C}:=\gamma\circ\overline{\eta} $ is
called the fundamental or the canonical (Liouville) vector field.
\par
Let $D$ be  a linear connection (or simply a connection) on the
pullback bundle $\pi^{-1}(TM)$.
 We associate with
$D$ the map \vspace{-0.1cm}
$$K:T \T M\longrightarrow \pi^{-1}(TM):X\longmapsto D_X \overline{\eta}
,\vspace{-0.1cm}$$ called the connection (or the deflection) map of
$D$. A tangent vector $X\in T_u (\T M)$ is said to be horizontal if
$K(X)=0$ . The vector space $H_u (\T M)= \{ X \in T_u (\T M) :
K(X)=0 \}$ of the horizontal vectors
 at $u \in  \T M$ is called the horizontal space to $M$ at $u$  .
   The connection $D$ is said to be regular if
\begin{equation}\label{direct sum}
T_u (\T M)=V_u (\T M)\oplus H_u (\T M) \qquad \forall u\in \T M .
\end{equation}
\par If $M$ is endowed with a regular connection, then the vector bundle
   maps
\begin{eqnarray*}
 \gamma &:& \pi^{-1}(T M)  \To V(\T M), \\
   \rho |_{H(\T M)}&:&H(\T M) \To \pi^{-1}(TM), \\
   K |_{V(\T M)}&:&V(\T M) \To \pi^{-1}(T M)
\end{eqnarray*}
 are vector bundle isomorphisms.
   Let us denote
 $\beta:=(\rho |_{H(\T M)})^{-1}$,
then \vspace{-0.2cm}
   \begin{align}\label{fh1}
    \rho\circ\beta = id_{\pi^{-1} (TM)}, \quad  \quad
       \beta\circ\rho =\left\{
                                \begin{array}{ll}
                                          id_{H(\T M)} & {\,\, on\,\,   H(\T M)} \\
                                         0 & {\,\, on \,\,   V(\T M)}
                                       \end{array}
                                     \right.\vspace{-0.2cm}
\end{align}
The map $\beta$ will be called the horizontal map of the connection
$D$.
\par According to the direct sum decomposition (\ref{direct
sum}), a regular connection $D$ gives rise to a horizontal projector
$h_{D}$ and a vertical projector $v_{D}$, given by
\begin{equation}\label{proj.}
h_{D}=\beta\circ\rho ,  \ \ \ \ \ \ \ \ \ \ \
v_{D}=I-\beta\circ\rho,
\end{equation}
where $I$ is the identity endomorphism on $T(TM)$: $I=id_{T(TM)}$.
\par
 The (classical)  torsion tensor $\textbf{T}$  of the connection
$D$ is defined by
$$\textbf{T}(X,Y)=D_X \rho Y-D_Y\rho X -\rho [X,Y] \quad
\forall\,X,Y\in \mathfrak{X} (\T M).$$ The horizontal ((h)h-) and
mixed ((h)hv-) torsion tensors, denoted by $Q $ and $ T $
respectively, are defined by \vspace{-0.2cm}
$$Q (\overline{X},\overline{Y})=\textbf{T}(\beta \overline{X}\beta \overline{Y}),
\, \,\, T(\overline{X},\overline{Y})=\textbf{T}(\gamma
\overline{X},\beta \overline{Y}) \quad \forall \,
\overline{X},\overline{Y}\in\mathfrak{X} (\pi (M)).\vspace{-0.2cm}$$
If $M$ is endowed with a metric $g$ on $\p$, we write
\begin{equation}\label{tor.g}
    T(\overline{X},\overline{Y},\overline{Z}):
=g(T(\overline{X},\overline{Y}),\overline{Z}).
\end{equation}
\par
The (classical) curvature tensor  $\textbf{K}$ of the connection $D$
is defined by
 $$ \textbf{K}(X,Y)\rho Z=-D_X D_Y \rho Z+D_Y D_X \rho Z+D_{[X,Y]}\rho Z
  \quad \forall\, X,Y, Z \in \mathfrak{X} (\T M).$$
The horizontal (h-), mixed (hv-) and vertical (v-) curvature
tensors, denoted by $R$, $P$ and $S$ respectively, are defined by
$$R(\overline{X},\overline{Y})\overline{Z}=\textbf{K}(\beta
\overline{X}\beta \overline{Y})\overline{Z},\quad
P(\overline{X},\overline{Y})\overline{Z}=\textbf{K}(\beta
\overline{X},\gamma \overline{Y})\overline{Z},\quad
S(\overline{X},\overline{Y})\overline{Z}=\textbf{K}(\gamma
\overline{X},\gamma \overline{Y})\overline{Z}.$$ The contracted
curvature tensors, denoted by $\widehat{R}$, $\widehat{P}$ and
$\widehat{S}$ respectively, are also known as the
 (v)h-, (v)hv- and (v)v-torsion tensors and are defined by
$$\widehat{R}(\overline{X},\overline{Y})={R}(\overline{X},\overline{Y})\overline{\eta},\quad
\widehat{P}(\overline{X},\overline{Y})={P}(\overline{X},\overline{Y})\overline{\eta},\quad
\widehat{S}(\overline{X},\overline{Y})={S}(\overline{X},\overline{Y})\overline{\eta}.$$
If $M$ is endowed with a metric $g$ on $\p$, we write
\begin{equation}\label{cur.g}
    R(\overline{X},\overline{Y},\overline{Z}, \overline{W}):
=g(R(\overline{X},\overline{Y})\overline{Z}, \overline{W}),\,
\cdots, \, S(\overline{X},\overline{Y},\overline{Z}, \overline{W}):
=g(S(\overline{X},\overline{Y})\overline{Z}, \overline{W}).
\end{equation}
\begin{equation}\label{.cur.g}
    \widehat{R}(\overline{X},\overline{Y},\overline{Z}):
=g(\widehat{R}(\overline{X},\overline{Y}), \overline{Z}),\cdots,\,
\widehat{S}(\overline{X},\overline{Y},\overline{Z}):
=g(\widehat{S}(\overline{X},\overline{Y}), \overline{Z}).
\end{equation}
\par We terminate this section
by some concepts and results concerning the Klein-Grifone approach
to intrinsic Finsler geometry. For more details, we refer to
\cite{r21}, \cite{r22} and \cite{r27}.
\par  A semispray  is a vector field $X$ on $TM$,
 $C^{\infty}$ on $\T M$, $C^{1}$ on $TM$, such that
$\rho\circ X = \overline{\eta}$. A semispray $X$ which is
homogeneous of degree $2$ in the directional argument
($[\mathcal{C},X]= X $) is called a spray.

\begin{prop}{\em{\cite{r27}}}\label{spray} Let $(M,L)$ be a Finsler manifold. The vector field
$G$ on $TM$ defined by $i_{G}\Omega =-dE$ is a spray, where
 $E:=\frac{1}{2}L^{2}$ is the energy function and $\Omega:=dd_{J}E$.
 Such a spray is called the canonical spray.
 \end{prop}

A nonlinear connection on $M$ is a vector $1$-form $\Gamma$ on $TM$,
$C^{\infty}$ on $\T M$, $C^{0}$ on $TM$, such that
$$J \Gamma=J, \quad\quad \Gamma J=-J .$$
The horizontal and vertical projectors $h_{\Gamma}$\,  and
$v_{\Gamma}$ associated with $\Gamma$ are defined by
   $h_{\Gamma}:=\frac{1}{2} (I+\Gamma)$ and  $v_{\Gamma}:=\frac{1}{2}
 (I-\Gamma).$

\begin{thm} \label{th.9a} {\em{\cite{r22}}} On a Finsler manifold $(M,L)$, there exists a unique
conservative homogenous nonlinear  connection  with zero torsion. It
is given by\,{\em:} \vspace{-0.3cm} $$\Gamma =
[J,G],\vspace{-0.3cm}$$ where $G$ is the canonical spray.\\
 Such a nonlinear connection is called the canonical connection, the Barthel connection or the Cartan nonlinear connection
 associated with $(M,L)$.
\end{thm}

%%%%%%%%%%%%%%%%%%%%%%%%%%%%%%%%%%%%%%% SECTION 2.  %%%%%%%%%%%%%%%%%%%%%%%%%%%%%%%%%

\Section{Fundamental identities associated with \vspace{5pt}regular
connections}

 In this section, we  investigate the most important
general properties concerning the torsion  and  curvature tensor
fields associated with regular connections  on $\pi^{-1}(TM)$. These
properties will play a key role throughout the whole paper.

\begin{defn}\emph{\cite{r92}} Let $D$ be a regular connection on
$\pi^{-1}(TM)$ with horizontal map~$\beta$.\\
\textbf{--} The semispray  $S=\beta\circ\overline{\eta}$ will be called the semispray associated with $D$.\\
\textbf{--} The nonlinear connection $\Gamma=2\beta\circ\rho-I$ will
be called the nonlinear connection associated with $D$ and will be
denoted by $\Gamma_{D}$.
\end{defn}

\begin{prop}\emph{\cite{r92}}\label{eqv.} Let  ${D}$ be a regular connection
on $\pi^{-1}(TM)$ whose connection map is $K$ and whose horizontal
map is $\beta$. Then, the following assertions are equivalent:
\vspace{-0.1cm}
 \begin{description}
    \item[(a)] The (h)hv-torsion  ${T}$ of ${D}$ has the property that
    ${T}( \overline{X},\overline{\eta})=0$,

    \item[(b)] $K=\gamma^{-1}$ on  $V(TM)$,

    \item[(c)] ${\Gamma} :=\beta\circ\rho - \gamma\circ K$ is a nonlinear
    connection on $M$.\vspace{-0.2cm}
 \end{description}
 \par Consequently,  if any one of the above assertions
    holds, then ${\Gamma}$
    coincides with the nonlinear connection associated with $D$\emph{:} ${\Gamma}=\Gamma_{D}=
    2\beta\circ\rho-I$, and in  this case $h_{{\Gamma}}=h_{D}=\beta\circ\rho$
    and  $\,v_{{\Gamma}}=v_{D}=\gamma\circ K$.
\end{prop}

The following two lemmas are fundamental for subsequent
use.\vspace{-0.2cm}
\begin{lem}\label{bracket} Let $D$ be a regular connection on $\p$
whose (h)hv-torsion tensor $T$ has the property that
$\,T(\overline{X}, \overline{\eta})=0$. Then, we
have{\em:}\vspace{-0.1cm}
   \begin{description}
 \item[(a)] $[\beta \overline{X},\beta \overline{Y}]=
     \gamma\widehat{R}(\overline{X},\overline{Y})
     + \beta(D_{\beta \overline{X}}\overline{Y}-
     D_{\beta \overline{Y}}\overline{X}-Q(\overline{X},\overline{Y})),$

    \item[(b)] $[\gamma \overline{X},\beta \overline{Y}]=-
     \gamma(\widehat{P}(\overline{Y},\overline{X})+D_
     {\beta \overline{Y}}\overline{X})
     +\beta( D_{\gamma \overline{X}}\overline{Y}-T(\overline{X},\overline{Y})),$

   \item[(c)] $[\gamma \overline{X},\gamma \overline{Y}]=
     \gamma(D_{\gamma \overline{X}}\overline{Y}-
     D_{\gamma \overline{Y}}\overline{X}+\widehat{S}(\overline{X},\overline{Y}))$.

     \end{description}
\end{lem}

\begin{proof}
 It should first be noted that, as $D$ is regular and
$T(\overline{X},\overline{\eta})=0$, we have $h=\beta\circ\rho$,
$\,v=\gamma\circ K$, $K\circ\gamma=id_{\cp} $, by Proposition
\ref{eqv.}. Then, we have
\begin{eqnarray*}
  [\beta \overline{X},\beta \overline{Y}] &=& \gamma
(K
  [\beta \overline{X},\beta \overline{Y}])
  +\beta(\rho [\beta \overline{X},\beta \overline{Y}]) =\gamma
(D_{
  [\beta \overline{X},\beta \overline{Y}]}\overline{\eta})+
  \beta(\rho [\beta \overline{X},\beta \overline{Y}]) \\
  &=&
     \gamma(\widehat{R}(\overline{X},\overline{Y})-D_
     {\beta \overline{Y}}D_
     {\beta\overline{X}}\overline{\eta}+D_
     {\beta\overline{X}}D_
     {\beta \overline{Y}}\overline{\eta})+\beta( D_{\beta
\overline{X}}\overline{Y}-D_{\beta
\overline{Y}}\overline{X}-Q(\overline{X},\overline{Y}))\\
&=&  \gamma\widehat{R}(\overline{X},\overline{Y})
     + \beta(D_{\beta \overline{X}}\overline{Y}-
     D_{\beta \overline{Y}}\overline{X}-Q(\overline{X},\overline{Y})).
\end{eqnarray*}

On the other hand,
\begin{eqnarray*}
  [\gamma \overline{X},\beta \overline{Y}] &=& \gamma
(K
  [\gamma \overline{X},\beta \overline{Y}])
  +\beta(\rho [\gamma \overline{X},\beta \overline{Y}]) = \gamma
(D_{
  [\gamma \overline{X},\beta \overline{Y}]}\overline{\eta})+
  \beta(\rho [\gamma \overline{X},\beta \overline{Y}]) \\
  &=&-
     \gamma(\widehat{P}(\overline{Y},\overline{X})+D_
     {\beta \overline{Y}}D_
     {\gamma\overline{X}}\overline{\eta}-D_
     {\gamma\overline{X}}D_
     {\beta \overline{Y}}\overline{\eta})+\beta( D_{\gamma
\overline{X}}\overline{Y}-T(\overline{X},\overline{Y}))\\
&=&-
     \gamma(\widehat{P}(\overline{Y},\overline{X})+D_
     {\beta \overline{Y}}\overline{X})
     +\beta( D_{\gamma \overline{X}}\overline{Y}-T(\overline{X},\overline{Y})).
\end{eqnarray*}

The last identity can be proved analogously.
\end{proof}

\begin{lem}\label{lem.2} Let $D$ be a linear connection on $\pi^{-1}(TM)$
with \emph{(}classical\emph{)} torsion  tensor $\textbf{T}$ and
\emph{(}classical\emph{)} curvature tensor $\textbf{K}$. For every
$X,Y,Z \in \cpp$, $\overline{Z},\overline{W}\in \cp$, we
have\,\emph{:}\vspace{-0.2cm}
\begin{description}
    \item[(a)]$\mathfrak{S}_{X,Y,Z}\{\textbf{K}(X,Y)\rho Z+D_{X}\textbf{T}(Y,Z)
    +\textbf{T}(X,[Y,Z])\}=0$,

 \item[(b)] $\mathfrak{S}_{X,Y,Z}\{D_{Z}\textbf{K}(X,Y)\overline{W}-\textbf{K}(X,Y)D_{Z}\overline{W}
    -\textbf{K}([X,Y],Z)\overline{W}\}=0$.
\end{description}
If  $\pi^{-1}(TM)$ is equipped with  a metric $g$,
then\vspace{-0.2cm}
\begin{description}
    \item[(c)]$g(\textbf{K}(X,Y)\overline{Z},\overline{W})+g(\textbf{K}(X,Y)\overline{W},\overline{Z})= \mathfrak{U}_{X,Y}\{
(D_{X}(D_{Y}g))(\overline{W},\overline{Z})\}-(D_{[X,Y]}g)(\overline{W},\overline{Z})$.
\end{description}
\end{lem}

\begin{proof} We prove (c) only.
$$ X\cdot g(\overline{W},\overline{Z} )=(D_{X}g)(\overline{W},\overline{Z})
+g(D_{X}\overline{W},\overline{Z})+g(\overline{W},D_{X}\overline{Z}).$$
From which, we obtain
\begin{eqnarray*}
   X\cdot ( Y\cdot g(\overline{W},\overline{Z} ))
   &=&  X\cdot((D_{Y}g)(\overline{W},\overline{Z}))
   + (D_{X}g)(D_{Y}\overline{W},\overline{Z})+
   (D_{X}g)(\overline{W},D_{Y}\overline{Z})
   \\
   &&+g(D_{X}D_{Y}\overline{W},\overline{Z})
   +g(D_{Y}\overline{W},D_{X}\overline{Z})
   +g(D_{X}\overline{W},D_{Y}\overline{Z})\\
   && +g(\overline{W},D_{X}D_{Y}\overline{Z}),
  \end{eqnarray*}
with similar expression for $Y\cdot ( X\cdot
g(\overline{W},\overline{Z} ))$. Consequently,
\begin{eqnarray*}
   [X,Y]\cdot g(\overline{W},\overline{Z} )&=&
    \mathfrak{U}_{X,Y}\{ X\cdot((D_{Y}g)(\overline{W},\overline{Z}))
   + (D_{X}g)(D_{Y}\overline{W},\overline{Z})+
   (D_{X}g)(\overline{W},D_{Y}\overline{Z})\}
   \\
   &&+g([D_{X},D_{Y}]\overline{W},\overline{Z})
      +g(\overline{W},[D_{X},D_{Y}]\overline{Z}).
  \end{eqnarray*}
 On the other hand, we have
\begin{eqnarray*}
   [X,Y]\cdot g(\overline{W},\overline{Z} )&=&(D_{[X,Y]}g)(\overline{W},\overline{Z})
+g(D_{[X,Y]}\overline{W},\overline{Z})+g(\overline{W},D_{[X,Y]}\overline{Z})
  \end{eqnarray*}
The result follows from the above two equations.
\end{proof}

\begin{prop}\label{pp.1} Let $D$ be a regular connection on $\p$
whose (h)hv-torsion tensor $T$ has the property that
$\,T(\overline{X}, \overline{\eta})=0$. Then, we
have{\em:}\vspace{-0.2cm}
\begin{description}
\item[(a)] $
S(\overline{X},\overline{Y})\overline{Z} =
 (D_{\gamma \overline{Y}}T)(\overline{X},\overline{Z})-
 (D_{\gamma \overline{X}}T)(\overline{Y},\overline{Z})$\\
 ${\,\qquad\qquad\,\,\,}+T(\overline{X},T(\overline{Y},\overline{Z}))
 -T(\overline{Y},T(\overline{X},\overline{Z}))+T( \widehat{S}(\overline{X},
 \overline{Y}),\overline{Z})$,

\item[(b)]$P(\overline{X},\overline{Y})\overline{Z}-P(\overline{Z},\overline{Y})\overline{X}=
(D_{\beta \overline{Z}}T)(\overline{Y},\overline{X})-(D_{\beta
\overline{X}}T)(\overline{Y},\overline{Z})-(D_{\gamma
\overline{Y}}Q)(\overline{X},\overline{Z})$\\
$ { \qquad\qquad\qquad \qquad\qquad\ \ }-T(\overline{Y},
Q(\overline{X},\overline{Z}))-T(\widehat{P}(\overline{Z},\overline{Y}),\overline{X})+
T(\widehat{P}(\overline{X},\overline{Y}),\overline{Z})$
\\ $ { \qquad\qquad\qquad \qquad\qquad\ \ }-Q(\overline{Z},
T(\overline{Y},\overline{X})) +Q(\overline{X},
T(\overline{Y},\overline{Z})),$

\item[(c)]$\mathfrak{S}_{\overline{X},\overline{Y},\overline{Z}}
\{R(\overline{X},
\overline{Y})\overline{Z}-T(\widehat{R}(\overline{X},\overline{Y}),\overline{Z})\}=
\mathfrak{S}_{\overline{X},\overline{Y},\overline{Z}}
\{Q(\overline{X},Q(\overline{Y},\overline{Z}))-(D_{\beta
\overline{X}}Q)(\overline{Y},\overline{Z})\}$.
\end{description}
\end{prop}

\begin{proof} Follows from Lemma
\ref{lem.2}{(a)} and  Lemma \ref{bracket}.
\end{proof}

\begin{prop}\label{pp.2} Let $D$ be a regular connection on $\p$
whose (h)hv-torsion tensor $T$ has the property that
$\,T(\overline{X}, \overline{\eta})=0$. Then, we
have{\em:}\vspace{-0.2cm}
\begin{description}
\item[(a)] $\mathfrak{S}_{\overline{X},\overline{Y},\overline{Z}}\{(D_{\gamma\overline{X}}S)
(\overline{Y},\overline{Z},\overline{W})-S(\widehat{S}(\overline{X},\overline{Y}),
\overline{Z})\overline{W}\}=0$.

\item[(b)] $(D_{\beta\overline{Z}}S)(\overline{X},\overline{Y},\overline{W}
)-(D_{\gamma
\overline{X}}P)(\overline{Z},\overline{Y},\overline{W})+ (D_{\gamma
\overline{Y}}P)(\overline{Z},\overline{X}, \overline{W})=$\\$=
P(T(\overline{X},\overline{Z}),\overline{Y})\overline{W}
-P(T(\overline{Y},\overline{Z}),\overline{X})\overline{W}
-P(\overline{Z},\widehat{S}(\overline{X},\overline{Y}))\overline{W}$\\
$+S(\widehat{P}(\overline{Z},\overline{X}),\overline{Y})\overline{W}
-S(\widehat{P}(\overline{Z},\overline{Y}),\overline{X})\overline{W}
$.

\item[(c)]$(D_{\gamma\overline{X}}R)(\overline{Y},\overline{Z},\overline{W})
   + (D_{\beta\overline{Y}}P)(\overline{Z},\overline{X},\overline{W})-
   (D_{\beta   \overline{Z}}P)(\overline{Y},\overline{X},\overline{W})=$\\
$=P(\overline{Z},\widehat{P}(\overline{Y},\overline{X}))\overline{W}
-P(\overline{Y},\widehat{P}(\overline{Z},\overline{X}))\overline{W}
- P(Q(\overline{Y},\overline{Z}),\overline{X})\overline{W}\\
+R(T(\overline{X},\overline{Z}),\overline{Y})\overline{W}
-R(T(\overline{X},\overline{Y}),\overline{Z})\overline{W}+
S(\widehat{R}(\overline{Y},\overline{Z}),\overline{X})\overline{W}$.

\item[(d)] $\mathfrak{S}_{\overline{X},\overline{Y},\overline{Z}}
\{(D_{\beta \overline{X}}R)(\overline{Y},
\overline{Z},\overline{W})+P(\overline{X},\widehat{R}(\overline{Y},\overline{Z}))\overline{W}+R(Q(\overline{X},\overline{Y}),\overline{Z})\overline{W}\}=0$.
\end{description}
\end{prop}

\begin{proof}
 Follows from Lemma \ref{lem.2}{(b)} and Lemma \ref{bracket}.
\end{proof}

%%%%%%%%%%%%%%%%%%%%%%%%%%%%%%%%%%%%%%%%%% Section. 3 %%%%%%%%%%%%%%%%%%%%%%%%%%%%%%%

\Section{Fundamental tensors associated with the Cartan
\vspace{7pt}\\
connection}

  We shall use the results obtained in $\S 2$ to investigate the fundamental properties of the
   most important tensors associated with Cartan connection.

\begin{thm}\emph{\cite{r92} }\label{th.1}Let $(M,L)$ be a Finsler manifold and  $g$  the Finsler metric
defined by $L$. There exists a unique regular connection $\nabla$ on
$\pi^{-1}(TM)$ such that\vspace{-0.2cm}
\begin{description}
  \item[(i)]  $\nabla$ is  metric\,{\em:} $\nabla g=0$,

  \item[(ii)] The (h)h-torsion of $\nabla$ vanishes\,{\em:} $Q=0
  $,
  \item[(iii)] The (h)hv-torsion $T$ of $\nabla$ satisfies \,
  $g(T(\overline{X},\overline{Y}), \overline{Z})=g(T(\overline{X},\overline{Z}),\overline{Y})$.\vspace{-0.2cm}
\end{description}
 \par
 Such a connection is called the Cartan
connection associated with  the Finsler manifold $(M,L)$.
\end{thm}

\begin{thm}\emph{\cite{r92}}\label{c.ba.}  The nonlinear connection associated with
the Cartan connection $\nabla$ coincides with the Barthel
connection{\em:} $\Gamma_{\nabla}=[J,G]$.
\end{thm}

\begin{prop}\label{pp.3} The (h)hv-torsion $T$ of the Cartan connection has the  properties\,:\vspace{-0.2cm}
\begin{description}
 \item[(a)] $T(\overline{X},\overline{Y},\overline{Z})=T(\overline{X},\overline{Z},\overline{Y})$,

 \item[(b)] $(\nabla_{W}T)(\overline{X},\overline{Y},\overline{Z})=g((\nabla_{W}T)(\overline{X},\overline{Y}),\overline{Z})=
 g((\nabla_{W}T)(\overline{X},\overline{Z}),\overline{Y})$,

\item[(c)] $T(\overline{X},\overline{\eta})=0$,

 \item[(d)] $ (\nabla_{\gamma \overline{X}}T)(\overline{Y},\overline{Z})=
 (\nabla_{\gamma \overline{Y}}T)(\overline{X},\overline{Z})$,

 \item[(e)]  $(\nabla_{\gamma\overline{\eta}}T)(\overline{X},\overline{Y})=-T(\overline{X},\overline{Y})$,

 \item[(f)] $T$ is totally symmetric.
\end{description}
\end{prop}

\begin{proof} ~\par
\vspace{4pt}
 \noindent\textbf{(b)} \ Follows from the following relations, making
use of (a):
$$
g((\nabla_{W}T)(\overline{X},\overline{Y}),\overline{Z})=
g(\nabla_{W}T(\overline{X},\overline{Y}),\overline{Z})-g(T(\nabla_{W}\overline{X},\overline{Y}),\overline{Z})
-g(T(\overline{X},\nabla_{W}\overline{Y}),\overline{Z}).
 $$
$$
g((\nabla_{W}T)(\overline{X},\overline{Z}),\overline{Y})=
g(\nabla_{W}T(\overline{X},\overline{Z}),\overline{Y})-
g(T(\nabla_{W}\overline{X},\overline{Z}),\overline{Y})-g(T(\overline{X},\nabla_{W}\overline{Z}),\overline{Y}).
$$
\begin{eqnarray*}
  g(\nabla_{W}T(\overline{X},\overline{Y}),\overline{Z})=W \cdot g(T(\overline{X},\overline{Y}),\overline{Z})-
  g(T(\overline{X},\overline{Y}),\nabla_{W}\overline{Z}).
\end{eqnarray*}
\vspace{4pt}
 \noindent\textbf{(c)} \ As  $\nabla$ is a metric linear connection
on $\pi^{-1}(T M)$ with
 nonzero torsion ${\textbf{T}}$, one can show that
 $\nabla$  is completely  determined  by the relation
\vspace{-0.2cm}
\begin{equation}\label{c1eq.r}
     \left.
    \begin{array}{rcl}
    2g(\nabla _{X}\rho Y,\rho Z)& =& X\cdot g(\rho Y,\rho
                Z)+ Y\cdot g(\rho Z,\rho X)-Z\cdot g(\rho X,\rho
                Y) \\
        & &-g(\rho X,{\textbf{T}}(Y,Z))+g(\rho Y,{\textbf{T}}(Z,X))
        +g(\rho Z,{\textbf{T}}(X,Y)) \\
        & &-g(\rho X,\rho [Y,Z])+g(\rho Y,\rho [Z,X])+g(\rho Z,\rho [X,Y]).\vspace{-0.2cm}
   \end{array}
  \right.
 \end{equation}
 for all $X,Y,Z\in\cpp$.
The connection $\nabla$ being regular, let  $h$ and $v$ be the
horizontal and vertical projectors associated with the
 decomposition (\ref{proj.}): $h=\beta\circ\rho$, $ v=I-\beta\circ\rho $.
\par
 Replacing $X, Y, Z$ by $\gamma \overline{X}, {h}Y, {h}Z $
 in (\ref{c1eq.r})
 and using  axiom  {(iii)} of Theorem \ref{th.1},  taking into account
  the fact that $\rho \circ \gamma=0$ and $\rho\circ h=\rho$, we get\vspace{-0.2cm}
\begin{equation}\label{eq}
2g(\nabla _{\gamma \overline{X}}\rho Y,\rho Z) =\gamma
\overline{X}\cdot g(\rho Y,\rho Z)+
     g(\rho Y,\rho [hZ,\gamma \overline{X}])+g(\rho Z,\rho [\gamma \overline{X},hY]).\vspace{-0.3cm}
\end{equation}
Now,  \vspace{-0.2cm}
\begin{eqnarray*}
  2\,g(T(\overline{X},\overline{\eta}),\overline{Z})=
  2\,g(\textbf{T}(\gamma\overline{X},\beta\overline{\eta}),\overline{Z})=
  2\,g(\nabla_{\gamma\overline{X}}\overline{\eta},\overline{Z})
  -2\,g(\rho[\gamma\overline{X}, \beta\overline{\eta}],\overline{Z}) \vspace{-0.2cm}.
\end{eqnarray*}
Then, from (\ref{eq}), we get\vspace{-0.2cm}
$$2\,g(T(\overline{X},\overline{\eta}),\overline{Z})=
\gamma\overline{X}\cdot g( \overline{\eta}, \overline{Z})+ g(
\overline{\eta},\rho [\beta\overline{Z},\gamma\overline{X}])-g(
\overline{Z},\rho [\gamma\overline{X},\beta\overline{\eta}]) .$$
 From which, together with the identity
 $\overline{X}= \rho\,[\gamma\overline{X},\beta\overline{\eta} ] \,$
  \cite{r92}, we
obtain\vspace{-0.2cm}
$$2\,g(T(\overline{X},\overline{\eta}),\overline{Z})=
\gamma\overline{X}\cdot g( \overline{\eta}, \overline{Z})+ g(
\overline{\eta},\rho [\beta\overline{Z},\gamma\overline{X}])-g(
\overline{Z},\overline{X}) \vspace{-0.2cm}.$$

 Finally, one can show that the sum of the first
two terms on the right-hand side is equal to $g(
\overline{X},\overline{Z})$, from which the result.

\vspace{4pt}
 \noindent\textbf{(d)}\, Since $\nabla$ is regular  with
$T(\overline{X},\overline{\eta})=0$,
  then, by Proposition \ref{pp.1}(a) and property (a), we have
\begin{equation*}\label{3.eq.7}
\left.
    \begin{array}{rcl}
S(\overline{X},\overline{Y},\overline{Z},\overline{W}) &=&
g((\nabla_{\gamma \overline{Y}}T)(\overline{X},\overline{Z}),
\overline{W}) -g((\nabla_{\gamma
\overline{X}}T)(\overline{Y},\overline{Z}),\overline{W})
\\
& &+g(T(\overline{X},\overline{W}),T(\overline{Y},\overline{Z}))-
g(T(\overline{Y},\overline{W}),T(\overline{X},\overline{Z}))\\
&&+g(T( \widehat{S}(\overline{X},
\overline{Y}),\overline{Z}),\overline{W}).\vspace{-0.2cm}
    \end{array}
  \right.
\end{equation*}
On the other hand, using Lemma \ref{lem.2}{(c)}, together with axiom
(i) of Theorem \ref{th.1}, we get
\begin{equation}\label{eq..1}
    S(\overline{X},\overline{Y},\overline{Z},\overline{W})
=-S(\overline{X},\overline{Y},\overline{W},\overline{Z}).
\end{equation}
 Using the properties (a) and (b), the above two
equations,  yield\vspace{-0.1cm}
\begin{equation}\label{3.eq.8}
    (\nabla_{\gamma \overline{X}}T)(\overline{Y},\overline{Z})-(\nabla_{\gamma
\overline{Y}}T)(\overline{X},\overline{Z})=T(
\widehat{S}(\overline{X},
\overline{Y}),\overline{Z}).\vspace{-0.1cm}
\end{equation}
Substituting (\ref{3.eq.8}) in (a) of Proposition \ref{pp.1}, we get
\begin{equation}\label{eq..2}
S(\overline{X},\overline{Y})\overline{Z}=
T(\overline{X},T(\overline{Y},\overline{Z}))-T(\overline{Y},T(\overline{X},\overline{Z})).\vspace{-0.2cm}
\end{equation}
Setting $\overline{Z}=\overline{\eta}$ in (\ref{eq..2}) and noting
that $T(\overline{X},
 \overline{\eta})=0$, we have
 \begin{equation}\label{eq.s.hat}
 \widehat{S}(\overline{X},\overline{Y})=0.
\end{equation}
Then,  the result follows from (\ref{3.eq.8}) and (\ref{eq.s.hat}) .

\vspace{4pt}
 \noindent\textbf{(f)}\, Follows from (d) by setting
$\overline{Z}=\overline{\eta}$, taking into account (c) and (a).
\end{proof}

%%%%%%%%%%%%%%%%%%%%%%%%%%%%%%%%%%%% [$ curvature  tensor] %%%%%%%%%%%%%%%%%%%%%%%%%%%%%%%%%%%

\begin{thm}\label{th..2} The v-curvature $S$ of the Cartan connection
has the properties\,\emph{:} \vspace{-0.2cm}
\begin{description}
 \item[(a)]   $ S(\overline{X},\overline{Y},\overline{Z}, \overline{W})=-
 S(\overline{Y},\overline{X},\overline{Z},\overline{W})$,

 \item[(b)]  $ S(\overline{X},\overline{Y},\overline{Z}, \overline{W})=-
 S(\overline{X},\overline{Y},\overline{W},\overline{Z})$,

\item[(c)]$S(\overline{X},\overline{Y})\overline{Z}=
T(\overline{X},T(\overline{Y},\overline{Z}))-T(\overline{Y},T(\overline{X},\overline{Z})),$

\item[(d)] $S(\overline{X}, \overline{Y} , \overline{Z}, \overline{W})
=g(T(\overline{X},\overline{W}) , T(\overline{Y}, \overline{Z}))-
g(T(\overline{Y}, \overline{W}) ,T( \overline{X}, \overline{Z}))$,

\item[(e)]  $S(\overline{Z},\overline{W},\overline{X}, \overline{Y})
=S(\overline{X},\overline{Y},\overline{Z},\overline{W})$,

 \item[(f)] $S(\overline{X}, \overline{\eta}) \overline{Y}=
 S(\overline{\eta}, \overline{X}) \overline{Y}={\widehat{S}}(\overline{X},\overline{Y})=0$,

   \item[(g)] $\mathfrak{S}_{\overline{X},\overline{Y},\overline{Z}}\,
 \{(\nabla_{\gamma \overline{X}}S)(\overline{Y}, \overline{Z}, \overline{W})\}=0,$

 \item[(h)] $\mathfrak{S}_{\overline{X},\overline{Y},\overline{Z}}\,
 \{S(\overline{X}, \overline{Y}) \overline{Z}\}=0$,

\item[(i)] $(\nabla_{\gamma \overline{\eta}}S)(\overline{X}, \overline{Y}, \overline{Z})=
-2S(\overline{X}, \overline{Y}) \overline{Z}$,

 \item[(j)] $(\nabla_{\beta\overline{Z}}S)(\overline{X},\overline{Y},\overline{W}
)=(\nabla_{\gamma
\overline{X}}P)(\overline{Z},\overline{Y},\overline{W})-
(\nabla_{\gamma \overline{Y}}P)(\overline{Z},\overline{X},
\overline{W})-
S(\widehat{P}(\overline{Z},\overline{Y}),\overline{X})\overline{W}$\\
${\qquad\qquad\qquad\qquad\quad}
+S(\widehat{P}(\overline{Z},\overline{X}),\overline{Y})\overline{W}
-P(T(\overline{Y},\overline{Z}),\overline{X})\overline{W}+
P(T(\overline{X},\overline{Z}),\overline{Y})\overline{W}$.
\end{description}
\end{thm}

\begin{proof}~\par

\vspace{4pt}
 \noindent\textbf{(b)},  \textbf{(c)}  and \textbf{(d)} follow immediately from  (\ref{eq..1}) and (\ref{eq..2}).

\vspace{4pt}
 \noindent\textbf{(e)} and  \textbf{(f)} follow from (d) and the properties  of $T$.

\vspace{4pt}
 \noindent\textbf{(g)} Follows from Proposition \ref{pp.2}(a) and  (\ref{eq.s.hat}).

\vspace{4pt}
 \noindent\textbf{(h)} and \textbf{(i)} follow from (g) by setting $\overline{W}=\overline{\eta}$
 and  $\overline{X}=\overline{\eta}$ respectively, taking (f) into
account.

\vspace{4pt}
 \noindent\textbf{(j)} \ Follows from Proposition \ref{pp.2}(b) and  (\ref{eq.s.hat}).
\end{proof}

%%%%%%%%%%%%%%%%%%%%%%%%%%%%%%%%% [P curvature  tensor] %%%%%%%%%%%%%%%%%%%%%%%%%%%%%%%

\begin{thm}\label{.thm1} The hv-curvature  tensor $P$ of the Cartan connection has
the properties\,:\vspace{-0.2cm}
\begin{description}
 \item[(a)]   $ P(\overline{X},\overline{Y},\overline{Z}, \overline{W})=-
 P(\overline{X},\overline{Y},\overline{W},\overline{Z})$,

\item[(b)] $P(\overline{X},\overline{Y})\overline{Z}-P(\overline{Z},\overline{Y})\overline{X}=
(\nabla_{\beta
\overline{Z}}T)(\overline{Y},\overline{X})-(\nabla_{\beta
\overline{X}}T)(\overline{Y},\overline{Z})$\\ $ { \qquad\qquad\qquad
\qquad\qquad\ \
}-T(\widehat{P}(\overline{Z},\overline{Y}),\overline{X})+
T(\widehat{P}(\overline{X},\overline{Y}),\overline{Z}),$

\item[(c)] $ P(\overline{X},\overline{Y},\overline{Z}, \overline{W}) =
   g(( \nabla_{\beta \overline{Z}}T)(\overline{X}, \overline{Y}),\overline{W})
   -g(( \nabla_{\beta \overline{W}}T)(\overline{X},
   \overline{Y}),\overline{Z})$\\
   ${\qquad\qquad\qquad \ \ \ \ }+ g(T(\overline{X},\overline{Z}),\widehat{P}(\overline{W},\overline{Y}))
   -g(T(\overline{X},\overline{W}),\widehat{P}(\overline{Z},
   \overline{Y})),$

 \item[(d)] $ \widehat{P}(\overline{\eta}, \overline{X})=0,$

  \item[(e)] $\widehat{P}(\overline{X}, \overline{Y})=(\nabla_{\beta \overline{\eta}}T)(\overline{X},\overline{Y}),$

\item[(f)] $\widehat{{P}}$  is  symmetric,

 \item[(g)] $ P(\overline{\eta}, \overline{X})\overline{Y}=P(\overline{X}, \overline{\eta})\overline{Y}=0,$

\item[(h)] $(\nabla_{\gamma \overline{\eta}}P)(\overline{X}, \overline{Y}, \overline{Z})=
-P(\overline{X}, \overline{Y}) \overline{Z}$,

 \item[(i)] $  P(\overline{X}, \overline{Y} ) \overline{Z}=
 P(\overline{Y}, \overline{X} ) \overline{Z} -
 (\nabla_{\beta \overline{\eta}}S)(\overline{X}, \overline{Y}, \overline{Z}).$
\end{description}
\end{thm}

\begin{proof}~\par

\vspace{4pt}
 \noindent\textbf{(a)}  Follows from Lemma \ref{lem.2}{(c)} by setting  $X=\beta\overline{X},
Y=\gamma\overline{Y}$, noting that $\nabla g=0$.

\vspace{4pt}
 \noindent\textbf{(b)} Follows from Proposition \ref{pp.1}(b) and  Theorem \ref{th.1}(ii).

\vspace{4pt}
 \noindent\textbf{(c)} From (b), making use of Proposition \ref{pp.3}{(a)}, we have\vspace{-0.2cm}
\begin{equation*}
\left.
    \begin{array}{rcl}
P(\overline{X},\overline{Y},\overline{Z},\overline{W})-
P(\overline{Z},\overline{Y},\overline{X}, \overline{W})&=&
g((\nabla_{\beta\overline{Z}}T)(\overline{Y},\overline{X}),
\overline{W})-g((\nabla_{\beta
\overline{X}}T)(\overline{Y},\overline{Z}), \overline{W})\\
&-&g(T(\overline{X},\overline{W}),\widehat{P}(\overline{Z},\overline{Y}))+
g(T(\overline{Z},\overline{W}),\widehat{P}(\overline{X},\overline{Y})).
    \end{array}
  \right.\vspace{-0.2cm}
\end{equation*}
By cyclic permutation on $\overline{X},\overline{Z},\overline{W}$ of
the above equation, one gets
\begin{equation*}
\left.
    \begin{array}{rcl}
P(\overline{W},\overline{Y},\overline{X},\overline{Z})-
P(\overline{X},\overline{Y},\overline{W}, \overline{Z})&=&
g((\nabla_{\beta\overline{X}}T)(\overline{Y},\overline{W}),
\overline{Z})-g((\nabla_{\beta
\overline{W}}T)(\overline{Y},\overline{X}), \overline{Z})\\
&-&g(T(\overline{W},\overline{Z}),\widehat{P}(\overline{X},\overline{Y}))+
g(T(\overline{X},\overline{Z}),\widehat{P}(\overline{W},\overline{Y})),
    \end{array}
  \right.\vspace{-0.3cm}
\end{equation*}
\begin{equation*}
\left.
    \begin{array}{rcl}
P(\overline{Z},\overline{Y},\overline{W},\overline{X})-
P(\overline{W},\overline{Y},\overline{Z}, \overline{X})&=&
g((\nabla_{\beta\overline{W}}T)(\overline{Y},\overline{Z}),
\overline{X})-g((\nabla_{\beta
\overline{Z}}T)(\overline{Y},\overline{W}), \overline{X})\\
&-&g(T(\overline{Z},\overline{X}),\widehat{P}(\overline{W},\overline{Y}))+
g(T(\overline{W},\overline{X}),\widehat{P}(\overline{Z},\overline{Y})).
    \end{array}
  \right.\vspace{-0.2cm}
\end{equation*}

\vspace{6pt} \noindent
Adding the first two equations and subtracting
the third, using (a) and Proposition \ref{pp.3}{(b)}, (f), the
result follows.

\vspace{4pt}
 \noindent\textbf{(d)} Follows from (c) by setting $\overline{X}=\overline{Z}=\overline{\eta}$,
making use of  the properties of $T$ and the fact that $K\circ
\beta=0$.

\vspace{4pt}
 \noindent\textbf{(e)} Follows from (c) by setting $\overline{Z}=\overline{\eta}$, taking (d) and the properties of $T$
into account.

\vspace{4pt}
 \noindent\textbf{(f)} Follows from (e) together with the symmetry of $T$.

\vspace{4pt}
 \noindent\textbf{(g)} Follows from (c) by setting $\overline{X}=\overline{\eta}$ (resp. $\overline{Y}=\overline{\eta}$),
making use of the  obtained properties of the (v)hv-torsion
 $\widehat{P}$ and the (h)hv-torsion  $T$.

\vspace{4pt}
 \noindent\textbf{(h)} Follows from the property (i) of the v-curvature tensor $S$ (Theorem
\ref{th..2})
 by setting $\overline{X}=\overline{\eta}$ and making use of the obtained
properties of  $T$, $S$ and $P$.

\vspace{4pt}
 \noindent\textbf{(i)} Can be proved in an analogous
manner as (h).
\end{proof}

%%%%%%%%%%%%%%%%%%%%%%%%%%%%%%%%% [R curvature  tensor]%%%%%%%%%%%%%%%%%%%%%%%%%%%%%%%

\begin{thm}\label{th.Rc}  The h-curvature  tensor $R$  of the Cartan connection has
the properties\,:\vspace{-0.2cm}
\begin{description}
 \item[(a)] $ R(\overline{X},\overline{Y},\overline{Z}, \overline{W})=
 -R(\overline{Y},\overline{X},\overline{Z},\overline{W})$,

 \item[(b)]  $ R(\overline{X},\overline{Y},\overline{Z}, \overline{W})=
 -R(\overline{X},\overline{Y},\overline{W},\overline{Z})$,

 \item[(c)] $ \widehat{R}(\overline{X}, \overline{Y})= - K\mathfrak{R}(\beta \overline{X},\beta \overline{Y})$,
  where $\mathfrak{R}$ is the curvature of Barthel
 connection,

\item[(d)] $\mathfrak{S}_{\overline{X},\overline{Y},\overline{Z}}\,
\{R(\overline{X},
\overline{Y})\overline{Z}-T(\widehat{R}(\overline{X},\overline{Y}),\overline{Z})\}=0,$

\item[(e)] $\mathfrak{S}_{\overline{X},\overline{Y},\overline{Z}}\,
\{(\nabla_{\beta \overline{X}}R)(\overline{Y},
\overline{Z},\overline{W})+P(\overline{X},\widehat{R}(\overline{Y},\overline{Z}))\overline{W}\}=0$,

\item[(f)] $(\nabla_{\gamma\overline{X}}R)(\overline{Y},\overline{Z},\overline{W})
   + (\nabla_{\beta\overline{Y}}P)(\overline{Z},\overline{X},\overline{W})-
   (\nabla_{\beta
   \overline{Z}}P)(\overline{Y},\overline{X},\overline{W})$\\
$-
P(\overline{Z},\widehat{P}(\overline{Y},\overline{X}))\overline{W}
+R(T(\overline{X},\overline{Y}),\overline{Z})\overline{W}-
S(\widehat{R}(\overline{Y},\overline{Z}),\overline{X})\overline{W}$\\
$+ P(\overline{Y},
\widehat{P}(\overline{Z},\overline{X}))\overline{W}
-R(T(\overline{X},\overline{Z}),\overline{Y})\overline{W}=0,$

\item[(g)] $(\nabla_{\gamma \overline{\eta}}R)(\overline{X}, \overline{Y},
\overline{Z})=0$.
\end{description}
\end{thm}

\begin{proof}~\par
\vspace{4pt}
 \noindent\textbf{(b)} Follows from Lemma \ref{lem.2}{(c)} by setting $X=\beta\overline{X}$ and $Y=\beta\overline{Y}$,
 taking into account the fact that $\nabla
 g=0$.

\vspace{4pt}
 \noindent\textbf{(c)} We use the identity $\mathfrak{R}(X,Y)=-v[hX,hY]$
 \cite{r97} together with  Lemma \ref{bracket}(a) and the fact
 that $K\circ \gamma=id_{\cp}$ and $K\circ \beta=0$:
  \begin{eqnarray*}
   v[h X, h Y]&=&\gamma \circ K[(\beta\circ\rho) X,(\beta\circ\rho)
   Y]\\
   &=& \gamma \circ K\{ \gamma\widehat{R}(\rho X,\rho Y)
     + \beta(D_{h X}\rho Y-
     D_{h Y}\rho X-Q(\rho X,\rho Y))\}\\
    &=& \gamma\widehat{R}(\rho X,\rho Y),
 \end{eqnarray*}
from which the result.

\vspace{4pt}
 \noindent\textbf{(d)}  Follows from Proposition \ref{pp.1}(c) and Theorem \ref{th.1}(b).

\vspace{4pt}
 \noindent\textbf{(e)} and \textbf{(f)}  follow from Proposition \ref{pp.2}(d) and (c) respectively,
 noting that $Q=0$.

\vspace{4pt}
 \noindent\textbf{(g)} Follows from (f), making use of the obtained
properties of $T$, $S$ and $P$.
\end{proof}

%%%%%%%%%%%%%%%%%%%%%%%%%%%%%%%%%%% Section. 6 %%%%%%%%%%%%%%%%%%%%%%%%%

\Section{Fundamental tensors associated\vspace{5pt} with the Berwald
connection}

In this section, we investigate  the fundamental properties of the
most important geometric objects associated with  Berwald
connection.

\vspace{5pt}
 \par
 The following theorem guarantees the existence
and uniqueness of the Berwald connection.\vspace{-0.2cm}
\begin{thm}\emph{\cite{r92}} \label{bth2.h2} Let $(M,L)$ be a Finsler manifold. There exists a
unique regular connection ${{D}}^{\circ}$ on $\pi^{-1}(TM)$ such
that
\begin{description}
 \item[(i)] $D^{\circ}_{h^{\circ}X}L=0$,
  \item[(ii)]   ${{D}}^{\circ}$ is torsion-free\,{\em:} ${\textbf{T}}^{\circ}=0 $,
  \item[(iii)]The (v)hv-torsion tensor $\widehat{P^{\circ}}$ of ${D}^{\circ}$ vanishes:
   $\widehat{P^{\circ}}(\overline{X},\overline{Y})= 0$.
  \end{description}
  \par Such a connection is called the Berwald
  connection associated with the Finsler manifold $(M,L)$.\\
  Moreover,  the nonlinear connection associated with
 the Berwald connection $D^{\circ}$ coincides with the Barthel connection{\,\em:}
$\Gamma_{D^{\circ}}=[J,G]$. Consequently, $\beta^{\circ}=\beta$ and
$K^{\circ}=K$.
\end{thm}

\begin{thm}\emph{\cite{r92}}\label{th.5} The Berwald connection $D^{\circ}$ is explicitly expressed in
terms of the Cartan connection $\nabla$ in the form:
 \vspace{-0.1cm}
  \begin{equation}\label{b10}
     {{D}}^{\circ}_{X}\overline{Y} = \nabla _{X}\overline{Y}
+{\widehat{P}}(\rho X,\overline{Y}) -T(K X,\overline{Y}).
\vspace{-0.2cm}
\end{equation}
In particular, we have\vspace{-0.1cm}
\begin{description}
  \item[(a)] $ {{D}}^{\circ}_{\gamma \overline{X}}\overline{Y}=\nabla _{\gamma
  \overline{X}}\overline{Y}-T(\overline{X},\overline{Y})$.

 \item[(b)] $ {{D}}^{\circ}_{\beta \overline{X}}\overline{Y}=\nabla _{\beta
  \overline{X}}\overline{Y}+\widehat{P}(\overline{X},\overline{Y}).$
\end{description}
\end{thm}

Concerning the metricity of the Berwald connection, we
have\vspace{-0.2cm}
\begin{lem}\label{le.3} For the Berwald connection  $D^{\circ}$, we have\vspace{-0.2cm}
\begin{description}
 \item[(a)]$(D^{\circ}_{\gamma \overline{X}}
g)(\overline{Y},\overline{Z})=2T(\overline{X},\overline{Y},\overline{Z})$,

 \item[(b)]$(D^{\circ}_{\beta \overline{X}}
g)(\overline{Y},\overline{Z})=-2\widehat{P}(\overline{X},\overline{Y},\overline{Z})$
\end{description}
\end{lem}

\begin{prop}\label{.pp} The v-curvature  $S^{\circ}$ of the Berwald connection vanishes\,\emph{:} $ S^{\circ}=0$.
\end{prop}

\begin{proof}
 Since $D^{\circ}$ is regular with $\textbf{T}^{\circ}=0$,
the result follows  from Proposition \ref{pp.1}(a).
\end{proof}

%%%%%%%%%%%%%%%%%%%%%%%%%%%%%%%%% [P^(o} curvature  tensor] %%%%%%%%%%%%%%%%%%%%%%%%%%%%%%%

\begin{thm}\label{th} The hv-curvature  tensor $P^{\circ}$ of the Berwald connection
 has the properties\,:\vspace{-0.2cm}
\begin{description}

\item[(a)] $\widehat{P^{\circ}}(\overline{X},\overline{Y})=0$,

 \item[(b)] $P^{\circ}(\overline{X},\overline{Y},\overline{Z},\overline{W})
 +P^{\circ}(\overline{X},\overline{Y},\overline{W},\overline{Z})=2(D^\circ_{\beta \overline{X}}T)(\overline{Y},\overline{Z},\overline{W})+
           2 (D^{\circ}_{\gamma
           \overline{Y}}\widehat{P})(\overline{X},\overline{Z},\overline{W})$,

\item[(c)] $P^{\circ}(\overline{X},\overline{Y})\overline{Z}=
P^{\circ}(\overline{Z},\overline{Y})\overline{X}$,

 \item[(d)] $(D^{\circ}_{\gamma \overline{X}}P^{\circ})(\overline{Y},
    \overline{Z},  \overline{W})
    =(D^{\circ}_{\gamma \overline{Z}}P^{\circ})(\overline{Y},
    \overline{X},  \overline{W})$,

\item[(e)]  $P^{\circ}$ is totally symmetric,

\item[(f)] $(D^{\circ}_{\gamma \overline{\eta}}P^{\circ})(\overline{X},
    \overline{Y},  \overline{Z})=-P^{\circ}(\overline{X}, \overline{Y})
    \overline{Y}$.
\end{description}
\end{thm}

\begin{proof}~\par

\vspace{4pt}
 \noindent\textbf{(b)} We successively use  Lemma \ref{lem.2}{(c)}
  (for $X=\beta\overline{X}, Y=\gamma\overline{Y}$), Lemma \ref{le.3}, Lemma \ref{bracket}(b)
  and finally Theorem \ref{bth2.h2}(iii). In fact,
     \begin{eqnarray*}
           P^{\circ}(\overline{X},\overline{Y},\overline{Z},\overline{W})
 +P^{\circ}(\overline{X},\overline{Y},\overline{W},\overline{Z}) &=&
 \beta \overline{X}\cdot(D^\circ_{\gamma \overline{Y}} g)(\overline{Z},\overline{W})
 -(D^\circ_{\gamma \overline{Y}} g)(D^{\circ}_{\beta \overline{X}}\overline{Z},\overline{W})\\
 &&- (D^\circ_{\gamma \overline{Y}} g)(\overline{Z},D^{\circ}_{\beta \overline{X}}\overline{W})-
 \gamma\overline{Y}\cdot(D^\circ_{\beta \overline{X}}
 g)(\overline{Z},\overline{W})\\
 && +(D^\circ_{\beta \overline{X}} g)(D^{\circ}_{\gamma \overline{Y}}\overline{Z},\overline{W})+
 (D^\circ_{\beta \overline{X}} g)(\overline{Z},D^{\circ}_{\gamma \overline{Y}}\overline{W})\\
 &&- (D^{\circ}_{[\beta \overline{X},\gamma \overline{Y}]}g)(\overline{Z},\overline{W}) \\
           &=&\beta \overline{X}\cdot(2T(\overline{Y},\overline{Z},\overline{W}))
 -2T(\overline{Y},D^{\circ}_{\beta \overline{X}}\overline{Z},\overline{W})\\
 &&- 2T(\overline{Y},\overline{Z},D^{\circ}_{\beta \overline{X}}\overline{W})-
 \gamma\overline{Y}\cdot(-2
 \widehat{P}(\overline{X},\overline{Z},\overline{W}))\\
 && -2\widehat{P}(\overline{X},D^{\circ}_{\gamma \overline{Y}}\overline{Z},\overline{W})
 -2 \widehat{P}(\overline{X},\overline{Z},D^{\circ}_{\gamma \overline{Y}}\overline{W})\\
 && -2T(D^\circ_{\beta \overline{X}}\overline{Y},\overline{Z},\overline{W})-2 \widehat{P}(D^\circ_{\gamma \overline{Y}}\overline{X} ,\overline{Z},\overline{W}) \\
           &=&2(D^\circ_{\beta \overline{X}}T)(\overline{Y},\overline{Z},\overline{W})+
           2 (D^{\circ}_{\gamma \overline{Y}}\widehat{P})(\overline{X},\overline{Z},\overline{W}).
          \end{eqnarray*}

\vspace{4pt}
 \noindent\textbf{(c)} and \textbf{(d)} follow from Proposition \ref{pp.1}(b) and
  Proposition \ref{pp.2}(b) respectively, taking  Proposition \ref{.pp} and the properties of $D^{\circ}$
  into account.

\vspace{4pt}
 \noindent\textbf{(e)} Follows from (d) by setting $\overline{W}=\overline{\eta}$, taking into
account (a) and (c).

\vspace{4pt}
 \noindent\textbf{(f)} \ Follows from (d) by setting
$\overline{X}=\overline{\eta}$ and using (a) and (e).
\end{proof}
%%%%%%%%%%%%%%%%%%%%%%%%%%%%%%%%% [R^{o} curvature  tensor] %%%%%%%%%%%%%%%%%%%%%%%%%%%%%%%

\begin{thm}\label{th1}The h-curvature tensor $R^{\circ}$  of the Berwald connection
has the properties\,:\vspace{-0.2cm}
\begin{description}
 \item[(a)]
 $R^{\circ}(\overline{X},\overline{Y},\overline{Z},\overline{W})=-
 R^{\circ}(\overline{Y},\overline{X},\overline{Z},\overline{W})$,

\item[(b)] $  \widehat{R^{\circ}}(\overline{X},
\overline{Y})= \widehat{R}(\overline{X},
\overline{Y})=-K\mathfrak{R}(\beta \overline{X},\beta
\overline{Y})$,

\item[(c)]
 $R^{\circ}(\overline{X},\overline{Y},\overline{Z},\overline{W})+
 R^{\circ}(\overline{X},\overline{Y},\overline{W},\overline{Z})=2\mathfrak{U}_{\overline{X},\overline{Y}}
 \{(D^\circ_{\beta \overline{Y}}\widehat{P})(\overline{X},\overline{Z},\overline{W})\}-2T(\widehat{R}(\overline{X},\overline{Y}),\overline{Z}, \overline{W})$,

 \item[(d)]  $\mathfrak{S}_{\overline{X},\overline{Y},\overline{Z}}\{
     R^{\circ}(\overline{X},\overline{Y}) \overline{Z}\}=0$,

  \item[(e)]  $\mathfrak{S}_{\overline{X},\overline{Y},\overline{Z}}\{
     (D^{\circ}_{\beta\overline{X}}R^{\circ})(\overline{Y},\overline{Z}, \overline{W})+
     P^{\circ}(\overline{X},
     \widehat{R}(\overline{Y},\overline{Z}))\overline{W}\}=0$,

\item[(f)] $(D^{\circ}_{\gamma \overline{X}}R^{\circ})(\overline{Y},
    \overline{Z},  \overline{W})=(D^{\circ}_{\beta \overline{Z}}P^{\circ})(\overline{Y},
    \overline{X},  \overline{W})-(D^{\circ}_{\beta \overline{Y}}P^{\circ})(\overline{Z},
    \overline{X},  \overline{W})$,

\item[(g)] $(D^{\circ}_{\gamma \overline{\eta}}R^{\circ})(\overline{X},
    \overline{Y},  \overline{Z})=0$,

\item[(h)] $\widehat{R^{\circ}}(\overline{X}, \overline{Y})=
    \frac{1}{3}\{(D^{\circ}_{\gamma\overline{X}}H)(\overline{Y})-
    (D^{\circ}_{\gamma\overline{Y}}H)(\overline{X})\}$;
    $H(\overline{X}):=\widehat{R^{\circ}}(\overline{\eta},
    \overline{X})$,

 \item[(i)] ${R}^{\circ}(\overline{X}, \overline{Y})\overline{Z}=
    (D^{\circ}_{\gamma\overline{Z}}\widehat{R^{\circ}})(\overline{X},\overline{Y})$.
\end{description}
\end{thm}

\begin{proof}~\par

\vspace{4pt}
 \noindent\textbf{(b)} Follows from Theorem \ref{th.Rc}(c)
  together with the identity $\widehat{R^{\circ}}=\widehat{R}$ \cite{r94}.

\vspace{4pt}
 \noindent\textbf{(c)} We use successively Lemma \ref{lem.2}{(c)}
  (for$X=\beta\overline{X}, Y=\beta\overline{Y}$), the property (b) above, Lemma \ref{le.3},
   Lemma \ref{bracket}(a)  and finally Theorem \ref{bth2.h2}(ii). In fact,
  \begin{eqnarray*}
            R^{\circ}(\overline{X},\overline{Y},\overline{Z},\overline{W})
 +R^{\circ}(\overline{X},\overline{Y},\overline{W},\overline{Z}) &=&
 \mathfrak{U}_{\overline{X},\overline{Y}}\{
 \beta \overline{X}\cdot(D^\circ_{\beta \overline{Y}} g)(\overline{Z},\overline{W})
 -(D^\circ_{\beta \overline{Y}} g)(D^{\circ}_{\beta \overline{X}}\overline{Z},\overline{W})\\
 &&- (D^\circ_{\beta \overline{Y}} g)(\overline{Z},D^{\circ}_{\beta \overline{X}}\overline{W})\}-
 (D^{\circ}_{[\beta \overline{X},\beta \overline{Y}]}g)(\overline{Z},\overline{W}) \\
           &=&\mathfrak{U}_{\overline{X},\overline{Y}}\{\beta \overline{X}\cdot
           (-2\widehat{P}(\overline{Y},\overline{Z},\overline{W}))
 +2\widehat{P}(\overline{Y},D^{\circ}_{\beta \overline{X}}\overline{Z},\overline{W})\\
 &&+ 2\widehat{P}(\overline{Y},\overline{Z},D^{\circ}_{\beta \overline{X}}\overline{W})\}-
 2T(\widehat{R}(\overline{X},\overline{Y}),\overline{Z},\overline{W})\\
 &&+2 \widehat{P}(D^\circ_{\beta \overline{X}}\overline{Y}-
 D^\circ_{\beta \overline{Y}}\overline{X} ,\overline{Z},\overline{W}) \\
           &=&2\mathfrak{U}_{\overline{X},\overline{Y}}
 \{(D^\circ_{\beta \overline{Y}}\widehat{P})(\overline{X},\overline{Z},\overline{W})\}-2T(\widehat{R}(\overline{X},\overline{Y}),\overline{Z}, \overline{W}).
 \end{eqnarray*}

\vspace{4pt}
 \noindent\textbf{(d)}  Follows from Proposition \ref{pp.1}(c), taking into
account the fact that $\textbf{T}^{\circ}=0$.

\vspace{4pt}
 \noindent\textbf{(e)}  Follows from Proposition \ref{pp.2}(d) together with the fact that $Q^{\circ}=0$
and $\widehat{R^{\circ}}=\widehat{R}$.

\vspace{4pt}
 \noindent\textbf{(f)}  Follows from Proposition \ref{pp.2}(c),
 noting that $\textbf{T}^{\circ}=\widehat{{P^{\circ}}}=S^\circ=0$.

\vspace{4pt}
 \noindent\textbf{(g)}  Follows from (f) by setting
$\overline{X}=\overline{\eta}$, making use of Theorem \ref{th}(a),
(e) and the fact that $K\circ \beta=0$.

\vspace{4pt}
 \noindent\textbf{(h)} We have, by (f) and Theorem \ref{th},
 \begin{equation}\label{HH}
(D^{\circ}_{\gamma \overline{X}}R^{\circ})(\overline{Y},
    \overline{Z},  \overline{\eta})
    =0.
\end{equation}
Setting $\overline{Y}=\overline{\eta}$ in (\ref{HH}),  noting that
$K\circ\gamma=id_{\cp}$, we get
\begin{eqnarray*}
  0&=&(D^{\circ}_{\gamma \overline{X}}R^{\circ})(\overline{\eta},
    \overline{Z},  \overline{\eta}) \\
    &=&D^{\circ}_{\gamma \overline{X}}R^{\circ}(\overline{\eta},
    \overline{Z})  \overline{\eta}-R^{\circ}(D^{\circ}_{\gamma \overline{X}}\overline{\eta},
    \overline{Z})  \overline{\eta}-R^{\circ}(\overline{\eta},
    D^{\circ}_{\gamma \overline{X}}\overline{Z})  \overline{\eta}-R^{\circ}(\overline{\eta},
    \overline{Z})  D^{\circ}_{\gamma \overline{X}}\overline{\eta} \\
  &=& D^{\circ}_{\gamma \overline{X}}H(\overline{Z})
  -\widehat{R^{\circ}}(\overline{X},\overline{Z}) -H( D^{\circ}_{\gamma \overline{X}}\overline{Z})
  -R^{\circ}(\overline{\eta},
    \overline{Z}) \overline{X} \\
  &=&(D^{\circ}_{\gamma \overline{X}}H)(\overline{Z})-\widehat{R^{\circ}}(\overline{X},\overline{Z})
  -R^{\circ}(\overline{\eta},
    \overline{Z}) \overline{X}    .
\end{eqnarray*}
From which, making use of (d), we obtain
\begin{eqnarray*}
  (D^{\circ}_{\gamma \overline{X}}H)(\overline{Y})-(D^{\circ}_{\gamma \overline{Y}}H)(\overline{X})
   &=& \widehat{R^{\circ}}(\overline{X},\overline{Y})
  +R^{\circ}(\overline{\eta},
    \overline{Y}) \overline{X}-\widehat{R^{\circ}}(\overline{Y},\overline{X})
  -R^{\circ}(\overline{\eta},
    \overline{X}) \overline{Y} \\
   &=& 2\widehat{R^{\circ}}(\overline{X},\overline{Y})
  -R^{\circ}(\overline{Y},\overline{\eta}) \overline{X}-R^{\circ}(\overline{\eta},
    \overline{X}) \overline{Y} =3\widehat{R^{\circ}}(\overline{X},\overline{Y}).
\end{eqnarray*}

\vspace{4pt}
 \noindent\textbf{(i)} From (\ref{HH}), noting that
$K\circ\gamma=id_{\cp}$, we get
    \begin{eqnarray*}
      0&=&(D^{\circ}_{\gamma \overline{X}}R^{\circ})(\overline{Y},
    \overline{Z},  \overline{\eta}) \\
    &=&D^{\circ}_{\gamma \overline{X}}R^{\circ}(\overline{Y},
    \overline{Z})  \overline{\eta}-R^{\circ}(D^{\circ}_{\gamma \overline{X}}\overline{Y},
    \overline{Z})  \overline{\eta}-R^{\circ}(\overline{Y},
    D^{\circ}_{\gamma \overline{X}}\overline{Z})  \overline{\eta}-R^{\circ}(\overline{Y},
    \overline{Z})  D^{\circ}_{\gamma \overline{X}}\overline{\eta} \\
      &=&(D^{\circ}_{\gamma \overline{X}}\widehat{R^{\circ}})(\overline{Y},
    \overline{Z})-R^{\circ}(\overline{Y},
    \overline{Z})\overline{X}.
    \end{eqnarray*}
This completes the proof.
\end{proof}

We terminate this section by the following\vspace{-0.2cm}
\begin{thm}The following assertion are equivalent\,:
\begin{description}
\item[(a)]  The curvature  tensor $\mathfrak{R}$  of Barthel connection  vanishes.

 \item[(b)] The h-curvature tensor ${R}^{\circ}$ of Berwald connection vanishes.

 \item[(c)] The (v)h-torsion tensor $\widehat{R^{\circ}}$ of Berwald connection vanishes.

  \item[(d)]The (v)h-torsion tensor $\widehat{R}$ of Cartan connection  vanishes.

\item[(e)] The $\pi$-tensor field $H$ vanishes.

 \item[(f)] The horizontal distribution is completely integrable.
\end{description}
\end{thm}

\begin{proof}
 These equivalences are realized by the properties (b), (h) and (i) of
 Theorem \ref{th1}, taking into account that
$\mathfrak{R}(X,Y)=-v[h X, h Y]$ \cite{r97}.
\end{proof}

%%%%%%%%%%%%%%%%%%%%%%%%%%%%%%%%%%% Section. 5 %%%%%%%%%%%%%%%%%%%%%%%%% %%%%%%%%%%%%%%%%%%%

\Section{Fundamental tensors  associated\vspace{5pt} with  the Chern
connection}

 In this section, we introduce and investigate the fundamental properties of the
   most important tensors associated with the Chern
connection.

\vspace{5pt}
 \par
 The following theorem guarantees the existence
and uniqueness of the Chern connection.\vspace{-0.2cm}
\begin{thm}\emph{\cite{r94}} \label{th.r1} Let $(M,L)$ be a Finsler manifold and $g$ the Finsler metric
defined by $L$. There exists a unique regular connection
$D^{\diamond}$ on $\pi^{-1}(TM)$ such that
\begin{description}
  \item[(i)]  $(D^{\diamond}_{X}\,g)(\rho Y,\rho Z)=2g(T(K^{\diamond} X,\rho Y), \rho
  Z)$,

  \item[(ii)]   $D^{\diamond}$ is torsion free\,\emph{:} $\textbf{T}^{\diamond}=0$,
\end{description}
\par
\noindent where $T$ is the (h)hv-torsion of the Cartan connection
and
$K^{\diamond}$ is the connection map of $D^{\diamond}$.\\
This connection is called the Chern (Rund) connection associated
with $(M,L)$.\\
Moreover, the nonlinear connection associated with the Chern
connection  $D^{\diamond}$ coincides with the Barthel
connection{\,\em:} $\Gamma_{D^{\diamond}}=[J,G]$. Consequently,
$\beta^{\diamond}= \beta$ and $k^{\diamond}=k$.
\end{thm}

\begin{thm}\emph{\cite{r94}}\label{th.r5}The Chern connection $D^{\diamond}$
is given in terms of the Cartan connection $\nabla$ \emph{(}or the
Berwald connection $D^{\circ}$\emph{)} by\,:
   $$ D^{\diamond}_{X}\overline{Y} = \nabla _{X}\overline{Y}
- T(KX,\overline{Y})= D^{\circ} _{X}\overline{Y} -{\widehat{P}}(\rho
X, \overline{Y}). \vspace{-0.2cm}$$ In particular, we have
\begin{description}
  \item[(a)] $ D^{\diamond}_{\gamma \overline{X}}\overline{Y}=\nabla _{\gamma
  \overline{X}}\overline{Y}-T(\overline{X},\overline{Y})=D^{\circ} _{\gamma
  \overline{X}}\overline{Y}$.

 \item[(b)] $ D^{\diamond}_{\beta \overline{X}}\overline{Y}=\nabla _{\beta
  \overline{X}}\overline{Y}=D^{\circ} _{\beta
  \overline{X}}\overline{Y}-{\widehat{P}}(\overline{X}, \overline{Y}).$
\end{description}
\end{thm}

Concerning the metricity of Chern connection, we have\vspace{-0.1cm}
\begin{lem}\label{le.4} For the Chern connection  $D^{\diamond}$, we have\vspace{-0.1cm}
\begin{description}
 \item[(a)]$(D^{\diamond}_{\gamma \overline{X}}
g)(\overline{Y},\overline{Z})=2T(\overline{X},\overline{Y},\overline{Z})$,

 \item[(b)]$D^{\diamond}_{\beta \overline{X}}
g=0$.
\end{description}
\end{lem}

\begin{lem} The v-curvature  $S^{\diamond}$ of the Chern connection vanishes\,\emph{:} $ S^{\diamond}=0$.
\end{lem}

\begin{proof} The proof is similar to that of proposition \ref{.pp}.
\end{proof}

\begin{thm}\label{.thm}
 The hv-curvature   $P^{\diamond}$ of the Chern connection  has the properties\,\emph{:}\vspace{-0.2cm}
\begin{description}
 \item[(a)] $P^{\diamond}(\overline{X},\overline{Y},\overline{Z},\overline{W})
    +P^{\diamond}(\overline{X},\overline{Y},\overline{W},\overline{Z})
    =2(D^{\diamond}_{\beta\,{\overline{X}}}T)(\overline{Y},\overline{Z},\overline{W})
    -2T(\widehat{P^{\diamond}}(\overline{X},\overline{Y}),
    \overline{Z},\overline{W})$,

\item[(b)]$P^{\diamond}(\overline{X},\overline{Y})\overline{Z}=P^{\diamond}(\overline{Z},\overline{Y})\overline{X}$,

\item[(c)] $ P^{\diamond}(\overline{X}, \overline{Y}, \overline{Z}, \overline{W})
=   (D^{\diamond}_{\beta\, \overline{X}}T)(\overline{Y},
\overline{Z}, \overline{W})
   +(D^{\diamond}_{\beta\, \overline{Z}}T)(\overline{Y}, \overline{W}, \overline{X})
   -(D^{\diamond}_{\beta\,\overline{W}}T)(\overline{Y}, \overline{X}, \overline{Z})$ \\
    ${\qquad\qquad\qquad\quad\,\,\,}+T(\widehat{P^{\diamond}}(\overline{W},\overline{Y}), \overline{X},\overline{Z})
    -T(\widehat{P^{\diamond}}(\overline{X},\overline{Y}), \overline{Z},\overline{W})
    -T(\widehat{P^{\diamond}}(\overline{Z},\overline{Y}),
    \overline{W},\overline{X})$,

 \item[(d)] $ \widehat{P^{\diamond}}(\overline{\eta}, \overline{X})=0,$

  \item[(e)] $\widehat{P^{\diamond}}(\overline{X}, \overline{Y})=\widehat{P}(\overline{X},
  \overline{Y})=
  (D^{\diamond}_{\beta \overline{\eta}}T)(\overline{X},\overline{Y}),$

\item[(f)] $\widehat{{P^{\diamond}}}$  is  symmetric,

 \item[(g)] $P^{\diamond}(\overline{X},
 \overline{\eta})\overline{Y}=0,$\quad $P^{\diamond}(  \overline{\eta},  \overline{X}
)\overline{Y}=(D^{\diamond}_{\beta\,
\overline{\eta}}T)(\overline{X}, \overline{Y})$,

\item[(h)] $(D^{\diamond}_{\gamma
\overline{X}}P^{\diamond})(\overline{Z},\overline{Y},\overline{W})=
(D^{\diamond}_{\gamma
\overline{Y}}P^{\diamond})(\overline{Z},\overline{X},
\overline{W})$,

\item[(i)] $(D^{\diamond}_{\gamma \overline{\eta}}P^{\diamond})(\overline{X}, \overline{Y}, \overline{Z})=
-P^{\diamond}(\overline{X}, \overline{Y}) \overline{Z}$.
\end{description}
\end{thm}

\begin{proof}~\par

\vspace{4pt}
 \noindent\textbf{(a)} By Lemma \ref{lem.2}(c), together with  Theorem \ref{th.r1}(i), we get
\begin{equation}\label{r.000}
\left.
    \begin{array}{rcl}
    g(\textbf{K}^{\diamond}(X,Y)\overline{Z},\overline{W})&+&g(\textbf{K}^{\diamond}(X,Y)\overline{W},\overline{Z})=\\
       &=&2\,\mathfrak{U}_{X,Y}\{
X\cdot g(T(K Y,\overline{W}),\overline{Z})
  + g(T(K X,D^{\diamond}_{Y}\overline{W}),\overline{Z})\\
  &&+
   g(T(K X,\overline{W}),D^{\diamond}_{Y}\overline{Z})\}
  -2g(T(K[X,Y],\overline{W}),\overline{Z}).
 \end{array}
  \right.
\end{equation}

From which, by setting  $X=\beta\, \overline{X}$ and $Y=\gamma
\overline{Y}$ in (\ref{r.000}) and using Lemma \ref{bracket}, noting
that  $\textbf{T}^{\diamond}=0$, we get
\begin{eqnarray*}
   P^{\diamond}(\overline{X},\overline{Y},\overline{Z},\overline{W})
   +P^{\diamond}(\overline{X},\overline{Y},\overline{W},\overline{Z})
&=&2\beta\,\overline{X}\cdot T(
\overline{Y},\overline{W},\overline{Z})
   - 2T(\overline{Y},D^{\diamond}_{\beta\,\overline{X}}\overline{W},\overline{Z}) \\
   &&- 2T(\overline{Y},\overline{W},D^{\diamond}_{\beta\,\overline{X}}\overline{Z})
  -2T(\widehat{P^{\diamond}}(\overline{X},\overline{Y}),\overline{W},\overline{Z})\\
  &&-2T( D^{\diamond}_{\beta\,\overline{X}}\overline{Y},\overline{W},\overline{Z}).
  \end{eqnarray*}
Hence, the result follows.

\vspace{4pt}
 \noindent\textbf{(b)} Follows from Proposition \ref{pp.1}(b),
     making use of the hypothesis that  $\textbf{T}^{\diamond}=0$.

\vspace{4pt}
 \noindent\textbf{(c)} Firstly, one can easily show that\vspace{-0.2cm}
 \begin{equation}\label{r.eq}
   (D^{\diamond}_{\beta\,{\overline{X}}}T)(\overline{Y},\overline{Z},\overline{W})=
g((D^{\diamond}_{\beta\,{\overline{X}}}T)(\overline{Y},\overline{Z}),\overline{W}).\vspace{-0.2cm}
 \end{equation}
Cyclic permutation on $\overline{X}, \overline{Z}, \overline{W}$ in
the formula given by (a) above yields three equations. Adding two of
these equations and subtracting the third, taking into account
(\ref{r.eq}) and  the property (b), gives\vspace{-0.2cm}
 \begin{equation}\label{r.01}
\left.
    \begin{array}{rcl}
   P^{\diamond}(\overline{X}, \overline{Y}, \overline{Z}, \overline{W}) &=&
   (D^{\diamond}_{\beta\, \overline{X}}T)(\overline{Y}, \overline{Z}, \overline{W})
   +(D^{\diamond}_{\beta\, \overline{Z}}T)(\overline{Y}, \overline{W}, \overline{X})
   -(D^{\diamond}_{\beta\,\overline{W}}T)(\overline{Y}, \overline{X}, \overline{Z}) \\
    &&+T(\widehat{P^{\diamond}}(\overline{W},\overline{Y}), \overline{X},\overline{Z})
    -T(\widehat{P^{\diamond}}(\overline{X},\overline{Y}), \overline{Z},\overline{W})
    -T(\widehat{P^{\diamond}}(\overline{Z},\overline{Y}),
    \overline{W},\overline{X}).
\end{array}
  \right.
\end{equation}

\vspace{4pt}
 \noindent\textbf{(d)} Follows from (c) by setting $\overline{X}=\overline{\eta}$ and
$\overline{Z}=\overline{\eta}$ and  making use of the properties of
the (h)hv-torsion $T$.

\vspace{4pt}
 \noindent\textbf{(e)} Follows from (c) by setting $\overline{Z}=\overline{\eta}$, taking
 (d), the properties of $T$ and the identity  $D^{\diamond}_{\beta\,
\overline{X}}\overline{Y}=
 \nabla_{\beta\, \overline{X}}\overline{Y}$ into account.

\vspace{4pt}
 \noindent\textbf{(f)} Follows from (e) and  the symmetry of $T$.

\vspace{4pt}
 \noindent\textbf{(g)} Follows from (c) by setting $\overline{Y}=\overline{\eta}$ (resp. $\overline{X}=\overline{\eta}$),
making use of the obtained  properties of $\widehat{P^{\diamond}}$
and $T$.

\vspace{4pt}
 \noindent\textbf{(h)} Follows from Proposition \ref{pp.2}(b), taking into account that
$S^{\diamond}= T^{\diamond}=0$.

\vspace{4pt}
\noindent\textbf{(i)} Follows from (h) by setting
$\overline{X}=\overline{\eta}$ and  making use of (g).
\end{proof}

\begin{cor}\label{.le.3}Let $(M,L)$ be a Finsler manifold. The following
assertion  are equivalent.\vspace{-0.2cm}
\begin{description}
  \item[(a)] The hv-curvature tensor $P$ vanishes\,: $P=0$,
  \item[(b)] The (v)hv-torsion  tensor $\widehat{P}$ vanishes\,:
  $\widehat{P}=0$.
\item[(c)] The (v)hv-torsion  tensor $\widehat{{P}^{\diamond}}$ vanishes\,:
  $\widehat{{P}^{\diamond}}=0$.
\end{description}
\end{cor}

\begin{proof}~\par

\vspace{4pt}
 \noindent\textbf{(a)$\Longrightarrow$(b)}: Trivial.

\vspace{4pt}
 \noindent (b)\textbf{$\Longrightarrow$(a)}:  Suppose that
$\widehat{P}$
 vanishes. From Theorem \ref{th..2}(i), we
 have\vspace{-0.2cm}
\begin{eqnarray*}
 (\nabla_{\beta\overline{Z}}S)(\overline{X},\overline{Y},\overline{W}
)&=&(\nabla_{\gamma
\overline{X}}P)(\overline{Z},\overline{Y},\overline{W})-
(\nabla_{\gamma \overline{Y}}P)(\overline{Z},\overline{X},
\overline{W})-
S(\widehat{P}(\overline{Z},\overline{Y}),\overline{X})\overline{W}\\
&&+S(\widehat{P}(\overline{Z},\overline{X}),\overline{Y})\overline{W}
-P(T(\overline{Y},\overline{Z}),\overline{X})\overline{W}+
P(T(\overline{X},\overline{Z}),\overline{Y})\overline{W}.\vspace{-0.2cm}
\end{eqnarray*}
Setting $ \overline{W}= \overline{\eta}$ in the above relation,
taking into account that $\widehat{S}=0$, we get\vspace{-0.2cm}
\begin{equation}\label{.eq.2}
   P(\overline{X},\overline{Y})\overline{Z}=P(\overline{X},\overline{Z})\overline{Y}.\vspace{-0.2cm}
\end{equation}

On the other  hand, from Theorem \ref{.thm1}(c), making use of the
given assumption, we have
\begin{equation}\label{.eq.3}
  P(\overline{X},\overline{Y},\overline{Z}, \overline{W}) =
   g(( \nabla_{\beta \overline{Z}}T)(\overline{X}, \overline{Y}),\overline{W})
   -g(( \nabla_{\beta \overline{W}}T)(\overline{X},\overline{Y}),\overline{Z}).
\end{equation}
From which, together with  (\ref{.eq.2}) and  $g(( \nabla_{\beta
\overline{W}}T)(\overline{X},\overline{Y}),\overline{Z})=g((
\nabla_{\beta
\overline{W}}T)(\overline{X},\overline{Z}),\overline{Y})$,  we
obtain
 $$( \nabla_{\beta
\overline{Z}}T)(\overline{X},\overline{Y})=( \nabla_{\beta
\overline{Y}}T)(\overline{X},\overline{Z})$$
 Now, again  from (\ref{.eq.3}) the result follows.

\vspace{4pt}
 \noindent  \textbf{(b)}$\Longleftrightarrow$\textbf{(c)}: Follows
from Theorem \ref{.thm}(e).
\end{proof}

%%%%%%%%%%%%%%%%%%%%%%%%%%%%%%%%% [R curvature  tensor]%%%%%%%%%%%%%%%%%%%%%%%%%%%%%%%

\begin{thm} The h-curvature  tensor $R^{\diamond}$ of the Chern connection
has the properties\,:\vspace{-0.2cm}

\begin{description}

 \item[(a)] $ R^{\diamond}(\overline{X},\overline{Y},\overline{Z}, \overline{W})=
 -R^{\diamond}(\overline{Y},\overline{X},\overline{Z},\overline{W})$,

\item[(b)] $ \widehat{R^{\diamond}}(\overline{X}, \overline{Y})=\widehat{R}(\overline{X}, \overline{Y})
 =- K\mathfrak{R}(\beta \overline{X},\beta \overline{Y})$,

 \item[(c)]  $ R^{\diamond}(\overline{X},\overline{Y},\overline{Z}, \overline{W})=
 -R^{\diamond}(\overline{X},\overline{Y},\overline{W},\overline{Z})-2
 T(\widehat{R}(\overline{X},\overline{Y}),\overline{Z}, \overline{W})$,

\item[(d)] $\mathfrak{S}_{\overline{X},\overline{Y},\overline{Z}}\,
\{R^{\diamond}(\overline{X}, \overline{Y})\overline{Z}\}=0,$

\item[(e)] $\mathfrak{S}_{\overline{X},\overline{Y},\overline{Z}}\,
\{(D^{\diamond}_{\beta \overline{X}}R^{\diamond})(\overline{Y},
\overline{Z},\overline{W})+P^{\diamond}(\overline{X},\widehat{R}(\overline{Y},\overline{Z}))\overline{W}\}=0$,

\item[(f)] $(D^{\diamond}_{\gamma\overline{X}}R^{\diamond})(\overline{Y},\overline{Z},\overline{W})
   + (D^{\diamond}_{\beta\overline{Y}}P^{\diamond})(\overline{Z},\overline{X},\overline{W})-
   (D^{\diamond}_{\beta
   \overline{Z}}P^{\diamond})(\overline{Y},\overline{X},\overline{W})$\\
$-
P^{\diamond}(\overline{Z},\widehat{P}(\overline{Y},\overline{X}))\overline{W}
+ P^{\diamond}(\overline{Y},
\widehat{P}(\overline{Z},\overline{X}))\overline{W}=0 ,$

\item[(g)] $(D^{\diamond}_{\gamma \overline{\eta}}R^{\diamond})(\overline{X}, \overline{Y},
\overline{Z})=0$.
\end{description}
\end{thm}

\begin{proof}~\par

\vspace{4pt}
 \noindent  \textbf{(b)} Follows from the identity $\widehat{R^{\diamond}}=\widehat{R}$ \cite{r94} together with
Theorem \ref{th.Rc}(c).

\vspace{4pt}
 \noindent  \textbf{(c)} Follows from (\ref{r.000}) by setting $X=\beta \overline{X}$ and $Y=\beta \overline{Y}$,
 making use of Lemma \ref{bracket}(a) and the identity
$\widehat{R^{\diamond}}=\widehat{R}$.

\vspace{4pt}
 \noindent  \textbf{(d)}  Follows from Proposition \ref{pp.1}(c), taking into
account the fact that $\textbf{T}^{\diamond}=0$.

\vspace{4pt}
 \noindent  \textbf{(e)}  Follows from Proposition \ref{pp.2}(d) together with $Q^{\diamond}=0$, making use of (c) above.

\vspace{4pt}
 \noindent  \textbf{(f)}  Follows from Proposition \ref{pp.2}(c), noting that
$\textbf{T}^{\diamond}=S^{\diamond}=0$ and
$\widehat{P^{\diamond}}=\widehat{P}$.

\vspace{4pt}
 \noindent  \textbf{(g)} Follows from (f) by setting
$\overline{X}=\overline{\eta}$, using the obtained properties of the
hv-curvature tensor $P^{\diamond}$.
\end{proof}

%%%%%%%%%%%%%%%%%%%%%%%%%%%%%%%%%%%%%%%%%%%% Section. 6 %%%%%%%%%%%%%%%%%%%%%%%%%%%%%%%%%%%

\Section{Fundamental tensors associated\vspace{5pt} with the
Hashiguchi connection}

 As in the previous section, we investigate
 the fundamental relations and properties of the most important tensors  associated with
the Hashiguchi connection.
\begin{thm}\emph{\cite{r94}} \label{th.h2} Let $(M,L)$ be a Finsler manifold and  $g$ the Finsler metric
defined by$\,L$. There exists a unique regular  connection ${D}^{*}$
on $\pi^{-1}(TM)$ such that
\begin{description}

  \item[(i)]  ${D}^{*}$ is vertically  metric\,{\em:} ${D}^{*}_{\gamma \overline{X}}\, g=0$,

   \item[(ii)]  The (h)hv-torsion  $T^{*}$ of ${D}^{*}$ satisfies\,{\em:}
  $g({T}^{*}(\overline{X},\overline{Y}), \overline{Z})=
  g({T}^{*}(\overline{X},\overline{Z}),\overline{Y})$,

 \item[(iii)]   The (h)h-torsion of ${D}^{*}$ vanishes\,{\em:} $Q^{*}=0
  $,

  \item[(iv)]The (v)hv-torsion of ${D}^{*}$ vanishes\,{\em:}
  $\widehat{{P}^{*}}=0$,

\item[(v)] $D^{*}_{h^{*}X}L=0$.
  \end{description}
   \par Such a connection is called the Hashiguchi connection associated with the
Finsler manifold $(M,L)$.
\end{thm}

\begin{thm}\emph{\cite{r94}} \label{h.ba.}  The nonlinear connection associated with
the  Hashiguchi connection  $D^{*}$ coincides with the Barthel
connection{\em:}
$\Gamma_{{D}^{*}}=[J,G]$. Consequently, $\beta^*=\beta$ and $K^*=K$.\\
Moreover, the (h)hv-torsion   of the Hashiguchi connection coincides
with the (h)hv-torsion of the Cartan connection\,\emph{:} $T^{*}=T$.
\end{thm}

\begin{thm}\emph{\cite{r94}}\label{th.t12} The Hashiguchi  connection $ {D}^{*} $  is
given in terms of the Cartan connection \emph{(}{\it or the Berwald
connection}\emph{)} by\,\emph{:} \vspace{-0.2cm}
  \begin{eqnarray}\label{4}
   {D}^{*}_{X}\overline{Y} = \nabla _{X}\overline{Y}
+ {\widehat{P}}(\rho X,\overline{Y}) ={D}^{\circ}_{X}\overline{Y}
+{T}(K X, \overline{Y}) . \vspace{-0.2cm}
  \end{eqnarray}
In particular, we have
\begin{description}
  \item[(a)] $ {D}^{*}_{\gamma \overline{X}}\overline{Y}=\nabla _{\gamma
  \overline{X}}\overline{Y}={D}^{\circ}_{\gamma \overline{X}}\overline{Y}
+{T}( \overline{X}, \overline{Y})$.

 \item[(b)] $ {D}^{*}_{\beta \overline{X}}\overline{Y}=\nabla _{\beta
  \overline{X}}\overline{Y}+\widehat{P}(\overline{X},\overline{Y})
  ={D}^{\circ}_{\beta \overline{X}}\overline{Y}.$
\end{description}
\end{thm}

Concerning the metricity of Hashiguchi connection, we
have\vspace{-0.1cm}
\begin{lem}\label{le.4} For the Hashiguchi connection  $D^{*}$, we have\vspace{-0.1cm}
\begin{description}
 \item[(a)]$D^{*}_{\gamma \overline{X}}
g=0$,

 \item[(b)]$(D^{*}_{\beta
\overline{X}}
g)(\overline{Y},\overline{Z})=-2g(\widehat{P}(\overline{X},\overline{Y}),\overline{Z})$.
\end{description}
\end{lem}

\begin{prop}\label{prop.Hs}The v-curvature $S^{*}$ of the Hashiguchi connection
 has the properties\,: \vspace{-0.2cm}
\begin{description}
 \item[(a)] $ S^{*}(\overline{X},\overline{Y},\overline{Z}, \overline{W})=-
 S^{*}(\overline{Y},\overline{X},\overline{Z},\overline{W})$,

 \item[(b)]  $ S^{*}(\overline{X},\overline{Y},\overline{Z}, \overline{W})=-
 S^{*}(\overline{X},\overline{Y},\overline{W},\overline{Z})$,

\item[(c)] $S^{*}(\overline{X}, \overline{Y} , \overline{Z},
\overline{W})=S(\overline{X}, \overline{Y} , \overline{Z},
\overline{W}) =g(T(\overline{X},\overline{W}) , T(\overline{Y},
\overline{Z}))- g(T(\overline{Y}, \overline{W}) ,T( \overline{X},
\overline{Z}))$,

   \item[(d)] $\mathfrak{S}_{\overline{X},\overline{Y},\overline{Z}}\,
 \{(D^{*}_{\gamma \overline{X}}S)(\overline{Y}, \overline{Z}, \overline{W})\}=0,$

\item[(e)] $(D^{*}_{\gamma \overline{\eta}}S)(\overline{X}, \overline{Y}, \overline{Z})=
-2S(\overline{X}, \overline{Y}) \overline{Z}$,

 \item[(f)] $(D^{*}_{\beta\overline{Z}}S)(\overline{X},\overline{Y},\overline{W}
)=(D^{*}_{\gamma
\overline{X}}P^{*})(\overline{Z},\overline{Y},\overline{W})-
(D^{*}_{\gamma \overline{Y}}P^{*})(\overline{Z},\overline{X},
\overline{W})
$\\
${\qquad\qquad\qquad\qquad}
-P^{*}(T(\overline{Y},\overline{Z}),\overline{X})\overline{W}+
P^{*}(T(\overline{X},\overline{Z}),\overline{Y})\overline{W}$.
\end{description}
\end{prop}

\begin{proof}~\par

\vspace{4pt}
 \noindent  \textbf{(b)}  Follows from Lemma \ref{lem.2}\textbf{(c)} by setting
$X=\gamma\overline{X}, Y=\gamma\overline{Y}$, taking into account
the fact that $D^{*}_{\gamma \overline{X}} g=0$ and making use of
Lemma \ref{bracket}(c) .

\vspace{4pt}
 \noindent  \textbf{(c)} Follows from the identity
$D^{*}_{\gamma \overline{X}}\overline{Y}=\nabla_{\gamma
\overline{X}}\overline{Y} $ (Theorem \ref{th.t12}), noting that the
vertical distribution is completely integrable.

\vspace{4pt}
 \noindent  \textbf{(d)} Since $D^*$ is regular with $T^*(\overline{X},\overline{\eta})=T(\overline{X},\overline{\eta})=0$, then
 from Proposition \ref{pp.2}(a) and the fact that $S^* =S$, the result follows.

\vspace{4pt}
 \noindent  \textbf{(e)}  Follows from (d) by setting $\overline{X}=\overline{\eta}$, taking into
account the property that
$S(\overline{X},\overline{\eta})\overline{Y}=
S(\overline{\eta},\overline{X})\overline{Y}=\widehat{S}(\overline{X},\overline{Y})=0$.

\vspace{4pt}
 \noindent  \textbf{(f)} Follows from Proposition \ref{pp.2}(b), noting that $S^* =S$ and
 $\widehat{P^{*}}=0$.
\end{proof}

\begin{thm} The hv-curvature  tensor $P^{*}$ of the Hashiguchi connection  has
the properties\,:\vspace{-0.2cm}
\begin{description}
 \item[(a)] $P^{*}(\overline{X},\overline{Y},\overline{Z},\overline{W})
 +P^{*}(\overline{X},\overline{Y},\overline{W},\overline{Z})=2
 (D^{*}_{\gamma \overline{Y}}\widehat{P})(\overline{X},\overline{Z},\overline{W})
 +2\widehat{P}(T(\overline{X},\overline{Y}),\overline{Z},\overline{W})$\,,

 \item[(b)] $ \widehat{P^{*}}=0,$
 \item[(c)] $P^{*}(\overline{X},\overline{Y})\overline{Z}-P^{*}(\overline{Z},\overline{Y})\overline{X}=
(D^{*}_{\beta
\overline{Z}}T)(\overline{Y},\overline{X})-(D^{*}_{\beta
\overline{X}}T)(\overline{Y},\overline{Z})$,

\item[(d)]$P^{*}(\overline{X},\overline{Y})\overline{Z}=P^{*}(\overline{X},\overline{Z})\overline{Y}$,

 \item[(e)] $\widehat{P^{*}}(\overline{\eta}, \overline{X}) \overline{Y}=
  -(D^{*}_{\beta
  \overline{\eta}}T)(\overline{X},\overline{Y}),\quad$ $\widehat{P^{*}}(\overline{X},\overline{\eta})
  \overline{Y}=0$,

\item[(f)] $(D^{*}_{\gamma \overline{\eta}}P^{*})(\overline{X}, \overline{Y}, \overline{Z})=
-P^{*}(\overline{X}, \overline{Y}) \overline{Z}$.

\end{description}
\end{thm}

\begin{proof}~\par

\vspace{4pt}
 \noindent  \textbf{(a)}  Follows from Lemma \ref{lem.2}\textbf{(c)} by setting
$X=\beta\overline{X}, Y=\gamma\overline{Y}$, using Lemma \ref{le.4}
and Lemma \ref{bracket}(c).
 In fact,
   \begin{eqnarray*}
            P^{*}(\overline{X},\overline{Y},\overline{Z},\overline{W})
 +P^{*}(\overline{X},\overline{Y},\overline{W},\overline{Z})
  &=&-\gamma\overline{Y}\cdot(D^*_{\beta \overline{X}} g)(\overline{Z},\overline{W})
 +(D^*_{\beta \overline{X}} g)(D^{*}_{\gamma \overline{Y}}\overline{Z},\overline{W})\\
 &&+(D^*_{\beta \overline{X}} g)(\overline{Z},D^{*}_{\gamma \overline{Y}}\overline{W})
 -(D^{*}_{[\beta \overline{X},\gamma \overline{Y}]}g)(\overline{Z},\overline{W}) \\
           &=& -\gamma\overline{Y}\cdot(-2 g(\widehat{P}(\overline{X},\overline{Z}),\overline{W}))
 -2 g(\widehat{P}(\overline{X},D^{*}_{\gamma \overline{Y}}\overline{Z}),\overline{W})\\
 &&-2 g(\widehat{P}(\overline{X},\overline{Z}),D^{*}_{\gamma \overline{Y}}\overline{W})
 -2 g(\widehat{P}(D^*_{\gamma \overline{Y}}\overline{X},\overline{Z}),\overline{W})\\
 &&+2 g(\widehat{P}(T(\overline{X},\overline{Y}),\overline{Z}),\overline{W})\\
           &=&2 (D^{*}_{\gamma \overline{Y}}\widehat{P})(\overline{X},\overline{Z},\overline{W})
           +2\widehat{P}(T(\overline{X},\overline{Y}),\overline{Z},\overline{W}).
          \end{eqnarray*}

\vspace{4pt}
 \noindent  \textbf{(b)}  Follows from Theorem \ref{th.h2}(d).

\vspace{4pt}
 \noindent  \textbf{(c)}  Follows from Proposition \ref{pp.1}(b), taking into
account the fact that $T^*=T$ and  $\widehat{P^{*}}=Q^*=0$.

\vspace{4pt}
 \noindent  \textbf{(d)}  Follows from Proposition \ref{prop.Hs}(f)
 by setting $\overline{W}=\overline{\eta}$, noting that
 $\widehat{S}=\widehat{P^{*}}=0$ and  $K\circ \beta=0$.

\vspace{4pt}
 \noindent  \textbf{(e)}  The first relation follows from (c) by setting $\overline{X}=\overline{\eta}$,
using  the obtained properties of the (h)hv-torsion $T$. On the
other hand, the relation $P^{*}(\overline{X},\overline{\eta})
  \overline{Y}=0$ follows from (d) by setting $\overline{Y}=\overline{\eta}$, making use of
  (b).

\vspace{4pt}
 \noindent  \textbf{(f)}   Follows from Proposition \ref{prop.Hs}(f) by setting
$\overline{X}=\overline{\eta}$, using (e)  and the obtained
properties of the v-curvature $S$.
\end{proof}

%%%%%%%%%%%%%%%%%%%%%%%%%%%%%%%%% [R curvature  tensor]%%%%%%%%%%%%%%%%%%%%%%%%%%%%%%%

\begin{thm} The h-curvature  tensor $R^{*}$ has
the properties\,:\vspace{-0.2cm}
\begin{description}
 \item[(a)] $ R^{*}(\overline{X},\overline{Y},\overline{Z}, \overline{W})
 =-R^{*}(\overline{Y},\overline{X},\overline{Z},\overline{W})$,

 \item[(b)]  $ R^{*}(\overline{X},\overline{Y},\overline{Z}, \overline{W})
 +R^{*}(\overline{X},\overline{Y},\overline{W},\overline{Z})=2
 \mathfrak{U}_{\overline{X},\overline{Y}}\{(D^{*}_{\beta \overline{Y}}\widehat{P})(\overline{X},\overline{Z},\overline{W})\}$,

  \item[(c)] $ \widehat{R}^{*}(\overline{X}, \overline{Y})= \widehat{R}(\overline{X},\overline{Y})=
  - K\mathfrak{R}(\beta \overline{X},\beta \overline{Y})$,

\item[(d)] $\mathfrak{S}_{\overline{X},\overline{Y},\overline{Z}}\,
\{R^{*}(\overline{X},
\overline{Y})\overline{Z}-T(\widehat{R}(\overline{X},\overline{Y}),\overline{Z})\}=0,$

\item[(e)] $\mathfrak{S}_{\overline{X},\overline{Y},\overline{Z}}\,
\{(D^{*}_{\beta \overline{X}}R^{*})(\overline{Y},
\overline{Z},\overline{W})+P^{*}(\overline{X},\widehat{R}(\overline{Y},\overline{Z}))\overline{W}\}=0$,

\item[(f)] $(D^*_{\gamma\overline{X}}R^{*})(\overline{Y},\overline{Z},\overline{W})
   + (D^{*}_{\beta\overline{Y}}P^{*})(\overline{Z},\overline{X},\overline{W})-
   (D^{*}_{\beta
   \overline{Z}}P^{*})(\overline{Y},\overline{X},\overline{W})$\\
$+R^{*}(T(\overline{X},\overline{Y}),\overline{Z})\overline{W}-
S(\widehat{R}(\overline{Y},\overline{Z}),\overline{X})\overline{W}
-R^{*}(T(\overline{X},\overline{Z}),\overline{Y})\overline{W}=0,$

\item[(g)] $(D^{*}_{\gamma \overline{\eta}}R^{*})(\overline{X}, \overline{Y},
\overline{Z})=0$.
\end{description}
\end{thm}

\begin{proof}~\par

\vspace{4pt}
 \noindent\textbf{(b)}  Follows from Lemma \ref{lem.2}\textbf{(c)} by setting
$X=\beta\overline{X}, Y=\beta\overline{Y}$, taking  Lemma \ref{le.4}
and Lemma \ref{bracket}(c) into account. In fact,
\begin{eqnarray*}
            R^{*}((\overline{X},\overline{Y},\overline{Z},\overline{W})
 +R^{*}(\overline{X},\overline{Y},\overline{W},\overline{Z})
 &=&\mathfrak{U}_{\overline{X},\overline{Y}}\{\beta \overline{X}\cdot(D^*_{\beta \overline{Y}} g)(\overline{Z},\overline{W})
 -(D^*_{\beta \overline{Y}} g)(D^{*}_{\beta \overline{X}}\overline{Z},\overline{W})\\
 &&-(D^*_{\beta \overline{Y}} g)(\overline{Z},D^{*}_{\beta \overline{X}}\overline{W})\}
 -(D^{*}_{[\beta \overline{X},\beta \overline{Y}]}g)(\overline{Z},\overline{W}) \\
           &=& \mathfrak{U}_{\overline{X},\overline{Y}}\{\beta \overline{X}\cdot(-2 g(\widehat{P}
           (\overline{Y},\overline{Z}),\overline{W}))
            +2 g(\widehat{P}(\overline{Y},D^{*}_{\beta \overline{X}}\overline{Z}),\overline{W})\\
 &&+2 g(\widehat{P}(\overline{Y},\overline{Z}),D^{*}_{\beta \overline{X}}\overline{W})\}
 +2 g(\widehat{P}(D^*_{\beta \overline{X}}\overline{Y},\overline{Z}),\overline{W}) \\
 &&-2 g(\widehat{P}(D^*_{\beta \overline{Y}}\overline{X},\overline{Z}),\overline{W}) \\
           &=&2\mathfrak{U}_{\overline{X},\overline{Y}}\{(D^{*}_{\beta \overline{Y}}P)(\overline{X},\overline{Z},\overline{W})\}.
          \end{eqnarray*}

\vspace{4pt}
 \noindent\textbf{(c)}  Follows from Theorem \ref{th.Rc}(c), taking into account that $\widehat{R^{*}}=\widehat{R}$ \cite{r94}.

\vspace{4pt}
 \noindent\textbf{(d)}  Follows from Proposition \ref{pp.1}(c), noting that $Q^{*}=0$, $\widehat{R^{*}}=\widehat{R}$ and
$T^*=T$ .

\vspace{4pt}
 \noindent\textbf{(e)}    Follows from Proposition \ref{pp.2}(d) together with $Q^*=0$ and
$\widehat{R^{*}}=\widehat{R}$.

\vspace{4pt}
 \noindent\textbf{(f)}  Follows from Proposition \ref{pp.2}(c), making use of the relations $S^*=S$,
$T^*=T$, $\widehat{R^{*}}=\widehat{R}$ and $\widehat{P^{*}}=Q^*=0$.

\vspace{4pt}
 \noindent\textbf{(g)}   Follows from (f) by setting
$\overline{X}=\overline{\eta}$ and using the obtained properties of
the (h)hv-torsion  $T$, the v-curvature $S$ and the hv-curvature
$P^{*}$.
\end{proof}

\vspace{9pt}
\newpage
\begin{center}
{\bf{\Large{ Appendix. Intrinsic Comparison}}}
\end{center}
\par The following tables establish a concise comparison concerning the
canonical linear connections in Finsler geometry as well as the
fundamental geometric objects associated with them.
 \begin{center}{\bf{Table 1.}}
\end{center}
 {\tiny
\begin{center}
\begin{tabular}
{|c|c|c|c|c|}\hline
&&&&\\
 {\bf connection} &{\bf Cartan:\,$\nabla$ }& {\bf Chern:\,$D^{\diamond}$ } &{\bf Hashiguchi:\,$D^{*}$ }
 &{\bf Berwald:\,$D^{\circ}$}
\\[0.1 cm]\hline
&&&&\\
{\bf v-counterpart} & $\nabla _{\gamma\overline{X}}\overline{Y}$&
 $D^{\diamond} _{\gamma\overline{X}}
      \overline{Y}=\nabla _{\gamma\overline{X}}
      \overline{Y}-T(\overline{X},\overline{Y})$&
${D}^{*} _{\gamma\overline{X}}
      \overline{Y}=\nabla _{\gamma\overline{X}}
      \overline{Y}$&
${D}^{\circ} _{\gamma\overline{X}}
      \overline{Y}=\nabla _{\gamma\overline{X}}
      \overline{Y}-T(\overline{X},\overline{Y})$
\\[0.1 cm]
{\bf h-counterpart} & $\nabla _{\beta\overline{X}}\overline{Y}$&
$D^{\diamond} _{\beta\overline{X}}
      \overline{Y}=\nabla _{\beta\overline{X}}
      \overline{Y}$&
${D}^{*} _{\beta\overline{X}}
      \overline{Y}= \nabla _{\beta\overline{X}}
      \overline{Y}+\widehat{P}(\overline{X},\overline{Y})$&
      ${D}^{\circ} _{\beta\overline{X}}
      \overline{Y}=\nabla _{\beta\overline{X}}
      \overline{Y}+\widehat{P}(\overline{X},\overline{Y})$

\\[0.1 cm]\hline
&&&&\\
{\bf (h)v-torsion} & $0$& $0$& $0$&$0$
\\[0.1 cm]
{\bf (h)hv-torsion} & $T$& $0$& $T$&$0$

\\[0.1 cm]
{\bf (h)h-torsion} & $0$& $0$&$0$&$0$
\\[0.1 cm]\hline
&&&&\\
{\bf (v)v-torsion} & $0$& $0$& $0$&$0$
\\[0.1 cm]
{\bf (v)hv-torsion}& $\widehat{P}=\nabla_{G}T$&
$\widehat{P}$&$0$&$0$
\\[0.1 cm]
{\bf (v)h-torsion} & ${\widehat{R}}=-K\mathfrak{R}$&
${\widehat{R}}$& ${\widehat{R}}$&${\widehat{R}}$
\\[0.1 cm]\hline
&&&&\\
{\bf v-curvature} & ${S}$& $0$& $S$&$0$
\\[0.1 cm]
{\bf hv-curvature} & ${P}$& ${P^{\diamond}}$&
${P^{*}}$&${P^{\circ}}$
\\[0.1 cm]
{\bf h-curvature} & ${R}$& $R^{\diamond}$& ${R^{*}}$&${R^{\circ}}$
\\[0.1 cm]\hline
&&&&\\
{\bf v-metricity}& $\nabla _{\gamma\overline{X}}{g}=0$&
$D^{\diamond} _{\gamma\overline{X}}{g}=2g(T(\overline{X},.),.)$&
$D^{*} _{\gamma\overline{X}}{g}=0$&$D^{\circ}
_{\gamma\overline{X}}{g}=2g(T(\overline{X},.),.)$
\\[0.1 cm]
{\bf h-metricity}& $\nabla _{\beta\overline{X}}{g}=0$& $D^{\diamond}
_{\beta\overline{X}}{g}=0$& $D^{*}
_{\beta\overline{X}}{g}=-2g(\widehat{P}(\overline{X},.),.)$&$D^{\circ}
_{\beta\overline{X}}{g}=-2g(\widehat{P}(\overline{X},.),.)$
\\[0.1 cm]\hline
\end{tabular}
\end{center}
}

\begin{center}{\bf{Table 2.}}
\\[0.2cm]
\begin{tabular}
{|c|c|}\hline
&\\
{\bf connection}&\bf{curvature tensors}
\\[0.1 cm]\hline
&\\
{\bf Cartan }&  v-curvature\,:\,
$S(\overline{X},\overline{Y})\overline{Z}:=-\nabla_{\gamma
\overline{X}} \nabla_{\gamma\overline{Y}}
\overline{Z}+\nabla_{\gamma \overline{Y}}
\nabla_{\gamma\overline{X}}
\overline{Z}+\nabla_{[\gamma\overline{X},\gamma\overline{X}]}
\overline{Z}$.
\\
\ \ & hv-curvature\,:\,
$P(\overline{X},\overline{Y})\overline{Z}:=-\nabla_{\beta
\overline{X}} \nabla_{\gamma\overline{Y}}
\overline{Z}+\nabla_{\gamma \overline{Y}} \nabla_{\beta\overline{X}
} \overline{Z}+\nabla_{[\beta\overline{X},\gamma\overline{X}]}
\overline{Z}$.
\\
\ \ & h-curvature\,:\,
$R(\overline{X},\overline{Y})\overline{Z}:=-\nabla_{\beta
\overline{X}} \nabla_{\beta\overline{Y}} \overline{Z}+\nabla_{\beta
\overline{Y}} \nabla_{\beta\overline{X} }
\overline{Z}+\nabla_{\beta\overline{X},\beta\overline{X}]}
\overline{Z}$.
\\[0.1 cm]\hline
&\\
{\bf Chern  }& $S^{\diamond}(\overline{X},\overline{Y})\overline{Z}=
0$.
\\
\ \ & $P^{\diamond}(\overline{X},\overline{Y})\overline{Z}=
  P(\overline{X},\overline{Y})\overline{Z}-T(\widehat{P}(\overline{X},\overline{Y}), \overline{Z})
  +(\nabla_{\beta \overline{X}}T)(\overline{Y},\overline{Z})$.
\\
\ \ &$R^{\diamond}(\overline{X},\overline{Y})\overline{Z}=
  R(\overline{X},\overline{Y})\overline{Z}-T(\widehat{R}(\overline{X},\overline{Y}),
  \overline{Z})$.
\\[0.1 cm]\hline
&\\
{\bf Hashiguchi }& ${S}^{*}(\overline{X},\overline{Y})\overline{Z}=
S(\overline{X},\overline{Y})\overline{Z}$.
\\
\ \ & ${P}^{*}(\overline{X},\overline{Y})\overline{Z}=
P(\overline{X},\overline{Y})\overline{Z}
  + \widehat{P}({{T}}(\overline{X},\overline{Y})
  , \overline{Z})+(\nabla_{\gamma \overline{Y}}\widehat{P})(\overline{X},\overline{Z})$.
\\
\ \ & ${R}^{*}(\overline{X},\overline{Y})\overline{Z}=
R(\overline{X},\overline{Y})\overline{Z}-
  \mathfrak{U}_{\overline{X},\overline{Y}}\{
   (\nabla_{\beta \overline{X}}\widehat{P})(\overline{Y}, \overline{Z})
  +\widehat{P}(\overline{X},\widehat{P}(\overline{Y},\overline{Z})) \}$.
\\[0.1 cm]\hline
&\\
{\bf Berwald}&
${{S}}^{\circ}(\overline{X},\overline{Y})\overline{Z}= 0 $.
\\
\ \ &${{P}}^{\circ}(\overline{X},\overline{Y})\overline{Z}=
P(\overline{X},\overline{Y})\overline{Z}+
  (\nabla_{\gamma \overline{Y}}\widehat{P})(\overline{X},\overline{Z})
  +\widehat{P}(T(\overline{Y},\overline{X}),\overline{Z})+$
  \\
\ \ &$ { \ \  \ \ \ }
   +\widehat{P}(\overline{X},T(\overline{Y},\overline{Z}))
   +(\nabla_{\beta\overline{X}}T)(\overline{Y},\overline{Z})-$
\\
\ \ & ${ \ \ \ } -T(\overline{Y},
\widehat{P}(\overline{X},\overline{Z}))-T( \widehat{P}(\overline{X},
\overline{Y}), \overline{Z})$.
\\
\ \ & ${{R}}^{\circ}(\overline{X},\overline{Y})\overline{Z}=
R(\overline{X},\overline{Y})\overline{Z}-
T(\widehat{R}(\overline{X},\overline{Y}),\overline{Z})-
\mathfrak{U}_{\overline{X},\overline{Y}}\{(\nabla_{\beta
\overline{X}}\widehat{P})(\overline{Y}, \overline{Z})$
\\
\ \ &$%{ \qquad  \ \ \ \ \ \ }
+\widehat{P}(\overline{X},\widehat{P}(\overline{Y},\overline{Z}))\}.$
\\[0.1 cm]\hline
\end{tabular}
\end{center}

\bigskip
%%%%%%%%%%%%%%%%%%%%%%% Refrence %%%%%%%%%%%%%%%%%%%%%%%%%%%%%%%%%%%%
%%%%%%%%%%%%%%%%%%%%%%% Refrence %%%%%%%%%%%%%%%%%%%%%%%%%%%%%%%%%%%%
\bigskip

\bigskip
\providecommand{\bysame}{\leavevmode\hbox
to3em{\hrulefill}\thinspace}
\providecommand{\MR}{\relax\ifhmode\unskip\space\fi MR }
\providecommand{\MRhref}[2]{%
  \href{http://www.ams.org/mathscinet-getitem?mr=#1}{#2}
} \providecommand{\href}[2]{#2}

\end{document}